\documentclass[final]{siamltex}
\usepackage{amsmath}
\usepackage{graphicx}
\usepackage[round]{natbib}

\title{Optimal approximation and Kolmogorov widths estimates
for certain singular classes related to equations of mathematical physics}
\author{Ilya V. Boykov
        \thanks{Department of High and Applied Mathematics, Penza State Univertsity, 
                40, Krasnaya, 440026, Penza, Russia. boikov@pnzgu.ru}
       }
\begin{document}
\newenvironment{rmk}{{\em Remark.}}{}
\maketitle

\begin{abstract}
{Solutions of numerous equations of mathematical physics such as elliptic, weakly singular,
singular, hypersingular integral equations belong to functional classes
$\bar Q^u_{r \gamma}(\Omega,1)$ and $Q^u_{r \gamma}(\Omega,1)$ 
defined over $l-$dimensional hypercube $\Omega=[-1,1]^l, l=1,2,\dots.$ The derivatives of
 classes' representatives grow indefinitely
when the argument approaches the boundary $\delta \Omega$.
In this paper we estimate the Kolmogorov and Babenko widths of two functional classes
$\bar Q^u_{r \gamma}(\Omega,1)$ and $Q^u_{r \gamma}(\Omega,1).$
We construct local splines belonging to those classes,
 such that the errors of approximation are
 of the same order as that of the estimated widths. 
 Thus we construct optimal with respect to order methods for
approximating the functional classes $\bar Q^u_{r \gamma}(\Omega,1)$ and $Q^u_{r \gamma}(\Omega,1).$}
 One can use these results for constructing  methods optimal with respect to  order for approximating a unit ball
 of the Sobolev spaces with logarithmic and polynomial weights.
\end{abstract}
\begin{keywords}
{Kolmogorov widths; Babenko widths; optimal approximation; splines.}
\end{keywords}


\section{Introduction}
\label{sec;introduction}

Let $B$ be a Banach space, $X \subset B$ be a compact set, and $\Pi: X \to
R_n$ be a mapping of $X \subset B$ onto a finite-dimensional space $R_n.$

\begin{definition}
 \rm\cite{[1]}. 
Let $L^n$ be  $n$-dimensional subspaces of the
linear space $B.$ The Kolmogorov width $d_n(X,B)$ is defined by
\begin{equation}
d_n(X,B) = \inf\limits_{L^n}\sup\limits_{x \in X}\inf\limits_{u \in L^n}
\|x-u\|,
\label{(1.1)}
\end{equation}
where the outer infimum is calculated over all $n-$dimensional subspaces of
$L^n.$ 
\label{Definition 1.1}
\end{definition}

\begin{definition} \rm\citep{[2],[3]}. The Babenko width $\delta_n(X)$ is defined by
\label{Definition 1.2} 
$
\delta_n(X) = \inf\limits_{\Pi:X \to R^n}\sup\limits_{x \in X}
\rm{diam}\Pi^{-1} \Pi(x),
$
where the infimum is calculated over all continuous mappings $\Pi:X \to R^n.$
\end{definition}

If the infimum in (\ref{(1.1)}) is attained for some $L^n,$ this subspace is called an
extremal subspace.

The widths evaluation for various spaces of functions play an
important role in numerical analysis and approximation theory since this problem is
 closely related to many optimality problems such as $\epsilon-$ complexity of
integration and approximation, optimal differentiation, and optimal
approximation of solutions of operator equations.

For a detailed study of these problems in view of the general theory of optimal algorithms
we refer to \citet{[4]}.  

\citet{[5]} formulated the problem of evaluating the widths $d_n(X,B),$   the discovery of extremal
subspaces of $L^n$.  \citet{[5]} also evaluated $d_n(X,B)$  for certain compact sets $X$.
Kolmogorov asserted to determine  the exact value of $d_n(X,B)$
 because it might lead to the discovery of extremal subspaces,
and therefore to new and better methods of approximation.  \citet{[2],[3]}
 promoted using extremal subspaces of compacts $X$ in constructing numerical methods
in physics and mechanics.

 The most general results were obtained in estimating the Kolmogorov widths in  Sobolev spaces $W^r_p$
on unit balls $B(W^r_p)$.
\citet{[6]} estimated the widths $d_n(B(W^r_1),L_2)$ and $d_n(B(W^r_\infty),L_\infty)$. 
\citet{[7]} obtained the exact values of $d_n(B(W^r_1),C)$. The widths
$d_n(B(W^r_p),L_q)$ for various $p$ and $q$ were studied 
by several authors, e.g. \citet{[11]}, \citet{[9],[10]}, \citet{[12]},\citet{[8]},\citet{[13]}. 
\citet{[12]} and \citet{[13]} obtained the final 
estimates of  $d_n(B(W^r_p),L_q)$ for $1 \le p \le \infty$ and $1 \le q \le \infty$.  
The widths of 
various classes of multivariable functions were analyzed by several scientists,
e.g.,\citet{[2]}, \citet{[3],[14]}, \citet{[19]}, \citet{[15]}, \citet{[17]}, \citet{[18]}, \citet{[16]},
 where the books \citet{[2]}, \citet{[19]}, \citet{[1]},\citet{[18]}, \citet{[16]},
and the articles \citet{[3],[17]}
may serve as reviews.

Solutions of numerous problems of  analysis, mechanics, electrodynamics,
and geophysics lead to the necessity to develop optimal methods for approximating 
special classes of functions. The classes $Q_{r\gamma}(\Omega,1)$ consist of 
functions having bounded derivatives up to  order $r$
in a closed domain $\Omega$ and higher order derivatives in $\Omega \backslash \delta\Omega,$
whose modulus increases unboundedly in a neighbourhood of the boundary $\partial\Omega$ 
(see Definitions \ref{Definition 1.3} -- \ref{Definition 1.7}).
The classes $Q_{r\gamma}(\Omega,1)$ describe solutions of elliptic equations 
\citet{[3]}, 
weakly singular, singular, and hypersingular integral equations \citet{[20]}. 

The functions representable by singular and hypersingular integrals with moving singularities 
$
\int\limits_{-1}^1\frac{\varphi (\tau)}{\tau-t}d\tau, \quad t\in (-1,1); \quad \quad
\int\limits_{-1}^1\frac{\varphi (\tau)}{(\tau-t)^p}d\tau, \quad t\in [-1,1], p=2,3,\cdots; 
$\\
$\int\limits_{-1}^1\int\limits_{-1}^1\frac{\varphi (\tau_1,\tau_2)}
{((\tau_1-t_1)^2+(\tau_2-t_2)^2)^{p/2}}d\tau_1d\tau_2, \quad t_1,t_2\in [-1,1], p=3,4,\cdots
$
also belong to $Q_{r\gamma}(\Omega,1)$ \citep[see][]{[21],[22]}.

 Apparently \citet{[3]}  defined the class of functions $Q_r(\Omega,1)$ 
to emphasize its role in construction approximation in numerous important problems in mathematical physics.  

The relationship between the functional class $Q_r(\Omega,1)$ (as well as 
$Q_{r\gamma}(\Omega,1)$) and compacts in the weighted Sobolev space 
$W^r_\infty (\Omega,1,\rho)$ follows from the definition of the classes.

Let $ \Omega = [-1,1]^l, l=1,2,\ldots,$ $t=(t_1, \ldots, t_l), r=(r_1, \ldots, r_l),$  
$r_i, i=1,\ldots,l,$ be  integers. Let
$\rho=d(t,\Gamma)$ be the $l_{\infty}$-distance between a point $t$ and the boundary
$\delta \Omega.$ 

The class $W_{\infty}^r(\Omega,1,\rho)$ consists of  functions $f \in C(\Omega),$ which have bounded 
partial derivatives of orders $i=0,1,\ldots,r$ in $\Omega$  and partial derivatives 
of orders $i=r+1,\ldots,2r+1$ in $\Omega \backslash \delta\Omega$ with the norm 
$\|f\| = \|f\|_{L_{\infty}(\Omega)}+  \\ \sum_{i=1}^{r}\sum_{i=i_1+ \cdots + i_l}
\|\partial^i
f  / \partial t^{i_1}_1  \cdots
\partial t^{i_l}_l \|_{L_{\infty}(\Omega)} + \\
\sum_{i=r+1}^{2r+1}\sum_{i=i_1+ \cdots + i_l}
\| (\rho(t))^{i-r}
\partial^i f  / \partial t^{i_1}_1  \cdots
\partial t^{i_l}_l \|_{L_{\infty}(\Omega)}\le 1, $ 
where $i_j$ are nonnegative integers, $0\le i_j \le i, j=1,\cdots,l,  i=
i_1+\cdots+i_l.$

Similarly one can define the classes of functions $W_{\infty}^{r\gamma}(\Omega,1,\rho),$  
$W_{\infty}^{r\gamma u}(\Omega,1,\rho),$ and
$\bar W_{\infty}^{r\gamma u}(\Omega,1,\rho)$
which are counterparts of the classes  
$Q_{r\gamma}(\Omega,1),$  $Q_{r\gamma}^u(\Omega,1),$ and $\bar Q_{r\gamma}^u(\Omega,1).$

The results of this paper can be extended to the classes 
$W_{\infty}^{r}(\Omega,1,\rho),$  $W_{\infty}^{r\gamma}(\Omega,1,\rho),$  
$W_{\infty}^{r\gamma u}(\Omega,1,\rho),$ and
$\bar W_{\infty}^{r\gamma u}(\Omega,1,\rho).$

The widths  estimates for the sets of functions $B(W_{\infty}^{r}(\Omega,1,\rho)),$  \ 
$B(W_{\infty}^{r\gamma}(\Omega,1,\rho)),$ \ \  
$B(W_{\infty}^{r\gamma u}(\Omega,1,\rho)),$ and
$B(\bar W_{\infty}^{r\gamma}(\Omega,1,\rho))$ are of interest since they play an important role in 
various applied problems, for example, problems of hydrodynamics. The author intends to use the
obtained results
in his further works in constructing optimal numerical methods for solving some problems of mathematical physics.

\section{Definitions of the function classes and previous results}
\label{sec;sec1}

\citet{[3]} defined the class $Q_r(\Omega,1)$ (Definition \ref{Definition 1.3}) and
 declared the problem of estimating  the Kolmogorov and Babenko widths of $Q_r(\Omega,1)$
to be one of the most important problems in the numerical analysis.
Later on this problem was solved by the author \citep[see][]{[23],[24]}.

The classes $Q_{r\gamma}(\Omega,1),$ $Q_{r\gamma p}(\Omega,1),$ and $B_{r\gamma}(\Omega,1)$ generalize the class \\ $Q_r(\Omega,1).$
\citet{[24],[19]} estimated  the Kolmogorov and Babenko widths 
and constructed local splines for approximation of functions from
$Q_{r\gamma}(\Omega,1),$  $Q_{r\gamma p}(\Omega,1),$ and $B_{r\gamma}(\Omega,1)$. 
The error of approximation obtained by local splines has the same order
as that of the corresponding values of the Kolmogorov and Babenko widths.
Below we list the definitions of the functional classes
$Q_{r\gamma}(\Omega,1),$ $Q_{r\gamma p}(\Omega,1),$ 
$Q^u_{r\gamma}(\Omega,1),$ and $\bar Q^u_{r\gamma}(\Omega,1).$

Let $\Omega = [-1,1]^l,$ $l\ge 1,$ $ \Gamma =\partial \Omega $ be the boundary of
$\Omega,$  and $u, r$ be positive integers. Let $t=(t_1,\ldots,t_l),$
$v=(v_1,\ldots,v_l),$  $|v| =v_1 + \cdots + v_l,$ $D^v=\partial^{|v|}/\partial t_1^{v_1}.\cdots.\partial t_l^{v_l}$
and $v_i$ be nonnegative
integers, $i=1, 2, \ldots,l.$

\begin{definition} \rm\citep{[3]}.
\label{Definition 1.3}  Let $\Omega = [-1,1]^l,$  $l=1, 2, \ldots.$
The class $Q_r(\Omega,1)$ consists of  functions $f\in C^r( \Omega)$ satisfying
$
\max\limits_{t \in \Omega}\left|D^vf(t)\right| \leq 1, \quad 0 \leq |v| \leq r,
$
$\left|D^vf(t)\right| \leq
(d(t,\Gamma))^{-(|v|-r)}, \quad t \in \Omega\setminus\Gamma,
\quad r < |v| \leq 2r+1,
$
where $d(t,\Gamma)$ is the $l_{\infty}$-distance between a point $t$ and
$\Gamma.$
\end{definition}

\begin{definition} \rm\citep{[24],[19]}.
\label{Definition 1.4} Let $\Omega = [-1,1]^l,$ $l=1, 2, \ldots.$
The class $Q_{r\gamma}(\Omega,1)$ consists of  functions $f\in C^r(\Omega)$
satisfying
$ 
\max\limits_{t \in \Omega}\left|D^vf(t)\right| \leq 1, \quad 0 \leq |v| \leq r,
$
$
\left|D^vf(t)\right| \leq
(d(t,\Gamma))^{-(|v|-r-\zeta)}, \quad  t \in \Omega\setminus\Gamma,
\quad r < |v| \leq s.
$
Note, $s=r+\lceil\gamma\rceil,$ $\zeta = \lceil\gamma\rceil-\gamma.$
\end{definition}

\begin{definition} \rm\citep{[24],[19]}.
\label{Definition 1.5} Let $\Omega = [-1,1]^l,$  $l=1, 2, \ldots.$
The class $Q_{r\gamma p}(\Omega,1)$ consists of  functions $f\in C^r(\Omega)$ 
 satisfying
$
\max\limits_{t \in \Omega} \left| D^vf(t) \right|  \leq 1, \quad 0 \leq |v| \leq r,
$
$
\int\limits_\Omega \left|(d(t,\Gamma))^{\gamma} D^vf(t) \right|^p dt \leq 1,
\quad r < |v| \leq s,
$
where $1\le p<\infty,$ $v=v_1+\cdots+v_l,$ $0 \le v_i \le s,$ $i=1,2,\ldots,l,$
 $s=r+\lceil\gamma\rceil$, $\zeta=\lceil\gamma\rceil-\gamma.$  
\end{definition}

\begin{definition}
\label{Definition 1.6} \rm Let $ \Omega = [-1,1]^l, l=1,2,\ldots.$ 
Let  $\gamma $ and $u$ be positive integers.
The class
$\bar Q^u_{r\gamma}(\Omega,1)$ consists of  functions $f \in C^{r-1}(\Omega)$
satisfying\\
$
\max\limits_{t \in \Omega}|D^vf(t) | \leq 1, \quad 0 \leq |v| \leq r-1,
$\\
$
|D^vf(t)| \leq (1+|\ln^u d(t,\Gamma)|), \quad   t \in \Omega\setminus\Gamma,
\quad |v| = r,
$\\
$
|D^vf(t)| \leq (1+|\ln^{u-1} d(t,\Gamma)|)/
(d(t,\Gamma))^{|v|-r}, \quad    t \in \Omega\setminus\Gamma,
\quad r < |v| \leq s,
$
where $s = r+\gamma.$
\end{definition}

\begin{definition}
\label{Definition 1.7} \rm Let $\Omega=[-1,1]^l, l=1,2,\ldots.$ Let $u$ be a positive integer, and
$\gamma$ be  a non-integer.
The class $Q^u_{r\gamma}(\Omega,1)$ consists of  functions $f\in C^r(\Omega)$
satisfying
$
\max_{t\in\Omega}|D^vf(t)|\leq 1,\ \ 0 \leq |v|\leq r,
$
$
|D^vf(t) |\leq (1+|\ln^ud(t,\Gamma)|)/(d(t,\Gamma))^{|v|-r-\zeta}
, \quad \ r<|v|\leq s, \quad  t \in \Omega\setminus\Gamma,
$
where $s=r+\lceil\gamma\rceil,\ \zeta=\lceil\gamma\rceil-\gamma.$
\end{definition}

\begin{definition}
\label{Definition 1.8} \rm Let $G = [a,b].$ The class $W^r(1),$ $r=1,2,\ldots,$
consists of  functions $f \in C[a,b]$ which have  absolutely
continuous derivatives of orders $j=0,1,\ldots,r-1$ and a
piecewise continuous derivative $f^{(r)}$ satisfying $|f^{(r)}|\leq
1.$
\end{definition}

\begin{definition}
\label{Definition 1.9} \rm Let $G=[a_1,b_1;\cdots;a_l,b_l]=[a_1,b_1]\times\cdots\times[a_l,b_l],$ 
$l=2,3,\ldots.$ The
class $C^r_l(1),$ $r=1,2,\cdots$ consists of  functions $f \in
C[G]$ which have  absolutely continuous partial derivatives of
orders $j=0,1,\ldots,r-1$ and a piecewise continuous partial derivative
of order $r$ satisfying $\max_{t\in G}\left|D^rf(t) \right| \leq 1.$ 
\end{definition}

Now we briefly describe the notations we use throughout
this paper.

Let $f \in W^r(1),$ $t \in [a,b],$ $c \in [a,b].$ 
We denote by $T_{r-1}(f,[a,b],c)$ the Taylor polynomial 
of $f$ of order $r-1$ with respect to the point $c,$ i.e. $ T_{r-1}(f,[a,b],c)= \sum^{r-1}_{j=0}(f^{(j)}(c) / j!) (t-c)^j.$
For $f(t_1,t_2)\in C^{r}_2(1),$ $t=(t_1,t_2)\in G=[a_1,b_1; a_2,b_2],$
$v=(v_1,v_2)\in G$ we can re-write its Taylor polynomial of order $r-1$ as
$T_{r-1}(f,G,v)=\sum^{r-1}_{j=0} d_j(f,v) / j!,$ where $d_j(f,v)$ is a polynomial of order $j$
given by $d_j(f,v)(t) = $\\$\sum^j_{i=0} C_{j}^i (\partial^j
f(v) / \partial t^i_1 \partial t^{j-i}_2) (t_1-v_1)^i
(t_2-v_2)^{j-i}.$ The Taylor polynomials for functions of $l\ge 2$
variables can be re-written similarly.

Many types of differential and integral equations
such as elliptic equations \citep{[3]},
weakly singular integral equations \citep{[29]}, 
singular, and hypersingular integral
equations \citep{[22],[30]}  have solutions that
belong to the functional sets similar to $\bar Q^u_{r \gamma}(\Omega,1),$ and
$Q^u_{r \gamma}(\Omega,1).$ Therefore  evaluating the widhts of 
$\bar Q_{r\gamma}^u(\Omega,1),$ and $Q_{r\gamma}^u(\Omega,1)$
and finding the extremal
subspaces are important problems in numerical analysis.

 The author \citep{[23],[24],[19]} developed 
optimal with respect to order to accuracy methods  for approximating the 
classes $Q_{r \gamma}(\Omega,1)$ and $B_{r \gamma}(\Omega,1)$.
Afterwards, these methods were used in developing (optimal with respect to order 
to accuracy and complexity) approximate methods for solving Fredholm and Volterra weakly singular integral 
equations \citep{[20]}, in estimating an accuracy of elliptic equations solutions \citep{[27]},
and in developing optimal with respect to accuracy cubature rules  for evaluating 
many-dimensional integrals \citep{[28]}. 
Solutions of some classes of weakly singular, singular, and hypersingular integral equations
belong to the classes $\bar Q^u_{r \gamma}(\Omega,1)$ and $Q^u_{r \gamma}(\Omega,1).$

In this paper we  obtain the weak asymptotic estimates of the Kolmogorov and Babenko
widths of the classes $\bar Q_{r\gamma}^u(\Omega,1)$ and $Q_{r\gamma}^u(\Omega,1).$  
We  also construct local splines, which yield the optimal order of 
approximation ( in the sense of the widths $d_n$ and $\delta_n$) 
for functions in $\bar Q_{r\gamma}^u(\Omega,1)$, $Q_{r\gamma}^u(\Omega,1).$

The author  intends to use the optimal methods for approximating the functional classes 
$\bar Q^u_{r \gamma}(\Omega,1)$ and $Q^u_{r \gamma}(\Omega,1)$ 
proposed in this paper for developing optimal methods of solving 
weakly singular, singular, and hypersingular integral equations.

We start with recalling the following well-known assertions.
\begin{lemma} \rm\citep{[1]}.
\label{Lemma 1.1}  Let $D$ be a Hausdorff space, $X \subset C(D).$
There exist $n+1$
points $t_i,$ $i=0,\ldots,n$  and a
number $\epsilon \geq 0$ with the following property:
For each distribution of signs $\lambda_i = \pm 1,$ $i=0,\ldots,n,$
there is a function $f_0 \in X$ such that
$sign \  f_0(t_i) = \lambda_i,$ $|f_0(t_i)| \geq \epsilon,$ $i=0,\ldots,n.$
Then  $d_n(X,C) \geq \epsilon$ in the space $C(D)$.   
\end{lemma}

\begin{lemma} \rm\citep{[2]}.
\label{Lemma 1.2} Let $B$ be a Banach space,  $X \subset B$ be a compact set. The inequality $\delta_n(X) \le 2d_n(X,B)$
is true.
\end{lemma}

Let $\zeta_k,$ $k=1,2,\ldots,r,$ be the zeros of the Chebyshev
polynomial of the first kind of degree $r.$
Moreover, let $f$  be a continuous function on $[-1,1],$ i.e. $f \in C([-1,1]).$
Finally, denote by  $L_r(f,[-1,1])$  the interpolating polynomial with respect to 
the Chebyshev nodes $\zeta_1,\ldots,\zeta_r.$

\begin{lemma} \rm\citep{[25]}.
\label{Lemma 1.3} If $f \in W^r,$ $r=1,2,\ldots,$ then \\
$
\|f- L_r(f,[-1,1])\|_{C[-1,1]} \leq \|f^{(r)}\|_{C[-1,1]}1/(r! 2^{r-1}).
$
\end{lemma}
Now we  briefly recall the  estimates of the Babenko and Kolmogorov widths
of the classes  $Q_{r\gamma}(\Omega,1),$ and $Q_{r\gamma p}(\Omega,1)$ with respect to 
$C$ and $L_q$.

If in the following theorems no restrictions on one or several of the parameters are given, then
the full range as described in the respective definitions is admissible.

\begin{theorem} \rm\citep{[24],[19]}.
\label{Theorem 1.1} Let $\Omega = [-1,1].$ Then
$
\delta_n(Q_{r\gamma}(\Omega,1))\asymp \\ \asymp d_n(Q_{r\gamma}(\Omega,1),C))\asymp n^{-s}.
$
\end{theorem}

\begin{theorem} \rm\citep{[24],[19]}.
\label{Theorem 1.2} Let $\Omega = [-1,1].$ If $\gamma$ is integer, then
\begin{equation*}
d_n(Q_{r\gamma p}(\Omega,1),L_q) \asymp 
\begin{cases}
n^{-s+1/p-1/q}, & 1 \leq p < q \leq 2, \\
n^{-s+1/p-1/2}, & 1 \leq p \leq 2, \  2<q <\infty,\\
n^{-s}, & 1\le q\le p < \infty, \  2\le p \le q < \infty.
\end{cases}
\end{equation*}
\end{theorem}
 
\begin{theorem} \rm\citep{[24],[19]}.
\label{Theorem 1.3} Let $\Omega = [-1,1]^l,$ $l \geq 2.$ Then
\begin{equation*}
\delta_n(Q_{r \gamma}(\Omega,1)) \asymp d_n(Q_{r \gamma}(\Omega,1),C) \asymp
\begin{array}{rl}
n^{-(s-\gamma)/(l-1)}, & v>l/(l-1),\\
n^{-s/l}, & v<l/(l-1),\\
n^{-s/l}(\ln n)^{s/l}, & v=l/(l-1),
\end{array}
\end{equation*}
where $v=s/(s-\gamma).$
\end{theorem}

\begin{theorem} \rm\citep{[24],[19]}.
\label{Theorem 1.4} Let $\Omega = [-1,1]^l,$ $l \geq 2,$ $1 \leq q \leq p <
\infty.$ If $\gamma$ is integer, then
\begin{equation*}
d_n(Q_{r \gamma p}(\Omega,1),L_q) \asymp
\begin{array}{rl}
n^{-r/(l-1)}, & v>l/(l-1),\\
n^{-s/l}, & v<l/(l-1),\\
(\ln n/n)^{s/l}, & v=l/(l-1),
\end{array}
\end{equation*}
where $v=s/(s-\gamma).$
\end{theorem}

\begin{theorem} \rm\citep{[24],[19]}.
\label{Theorem 1.5} Let $\Omega = [-1,1]^l,$ $l \geq 2,$ $1 \leq p < q \leq 2.$ 
If $\gamma$ is integer, then
\begin{equation*}
d_n(Q_{r \gamma p}(\Omega,1),L_q) \asymp
\begin{cases}
n^{-(r-l/p+l/q)/(l-1)}, & v>l/(l-1),\\
n^{-(s-l/p+l/q)/l}, & v<l/(l-1),\\
(\ln n/n)^{(s-l/p+l/q)/l}, & v=l/(l-1),
\end{cases}
\end{equation*}
where $v=(s-l/p+l/q)/(s-l/p+l/q-\gamma).$
\end{theorem}

\begin{theorem} \rm\citep{[24],[19]}.
\label{Theorem 1.6} Let $\Omega = [-1,1]^l,$ $l \geq 2,$ $1 \leq p \leq 2, 2< q < \infty.$ 
If $\gamma$ is integer, then
\begin{equation*}
d_n(Q_{r \gamma p}(\Omega,1),L_q) \asymp
\begin{cases}
n^{-(r-l/p+l/q)/(l-1)+1/q-1/2}, & v>l/(l-1),\\
n^{-(s/l-1/p+1/2)}, & v<l/(l-1),
\end{cases}
\end{equation*}
where $v=(s-l/p+l/q)/(s-l/p+l/q-\gamma).$
\end{theorem}

\begin{rmk}
The articles \citep{[24],[19]} also contain explicit constructions of local splines
which yield the optimal order of approximation error in Theorems \ref{Theorem 1.1}--\ref{Theorem 1.6}.
 Hence these splines 
can be regarded  as optimal methods of approximation in the sense of the Kolmogorov 
and Babenko widths. 
\end{rmk}

In this paper we extend some of these results
to the classes $\bar Q^u_{r\gamma}([-1,1]^l,1)$ and $Q^u_{r\gamma}([-1,1]^l,1),$ $l \ge 1.$

\section{Widths of the classes $\bar Q^u_{r\gamma}([-1,1],1)$ and
$Q^u_{r\gamma}([-1,1],1)$ of functions of one variable}
\label{sec;Sec2}

In this section we estimate the Kolmogorov and Babenko widths 
for each of the functional classes
$\bar Q^u_{r \gamma}(\Omega,1)$ and $Q^u_{r \gamma}(\Omega,1),$
$\Omega =[-1,1].$

\begin{theorem} \rm
\label{Theorem 2.1} 
Let $\Omega = [-1,1].$ Let $ r, u,\gamma$ be
positive integers, $s=r+\gamma.$ Then
$
\delta_n(\bar Q^u_{r \gamma}(\Omega,1)) \asymp
d_n(\bar Q^u_{r \gamma}(\Omega,1),C) \asymp n^{-s}.
$
\end{theorem}

\begin{proof}
First we estimate the infimum for $\delta_n(\bar Q^u_{r \gamma}(\Omega,1)).$ 
Note $Q_{r \gamma}(\Omega,1) \subset \\ \bar Q^u_{r \gamma}(\Omega,1).$
By Theorem \ref{Theorem 1.1} we know
$\delta_n(Q_{r \gamma}(\Omega,1)) \asymp n^{-s}.$ Therefore
\begin{equation}
\delta_n(\bar Q^u_{r \gamma}(\Omega,1)) \geq
\delta_n(Q_{r \gamma}(\Omega,1)) \asymp n^{-s}.
\label{(2.1)}
\end{equation}

To construct a continuous local spline with $n$ parameters that
approximates the functions of $\bar Q^u_{r \gamma}(\Omega,1)$ with
the accuracy $An^{-s}$, we study two cases:
i) $u=1,$ and ii) $u>1.$

{ i).} Let $u=1.$  We divide the interval $[-1,1]$ into $2N$
subintervals by the points $t_k = -1+(k/N)^v$ and $\tau_k = 1-(k/N)^v,$
 $k=0,1,\ldots,N,$  $v= s/(s-\gamma).$
Then divide the obtained interval $[t_0, t_1]$ into
$M (M=\lceil\ln N\rceil)$ subintervals by the points $t_{0,j} = t_0 +(t_1-t_0)j/M,$
$j =0,1,\ldots,M.$ Continuing the partition procedure we subdivide $[\tau_1, \tau_0]$
into $M$ subintervals by $\tau_{0,j} = \tau_0 - (\tau_0 - \tau_1)j/M,$
$j=0,1,\ldots,M.$

For each interval $[a,b]$ we now choose a polynomial $P_s(f,[a,b])$
interpolating $f(t) \in \bar Q^u_{r\gamma}(\Omega,1)$ at the endpoints $a$ and $b$ 
in the following way.
Denote  the zeros of the Chebyshev polynomial of the first kind
of degree $s$ by $\zeta_k,$ $k=1,2,\ldots,s.$ Map
$[\zeta_1,\zeta_s] \subset [-1,1]$  with an affine-linear  transformation 
onto $[a,b]$ so that the points $\zeta_1$
and $\zeta_s$ are mapped onto $a$ and $b$ respectively.
Denote the images of the points $\zeta_i$ under this mapping by
$\zeta'_i,$ $i=1,2,\ldots,s.$
We then denote by $P_s(f,[a,b])$ the interpolation polynomial of degree $s-1$ 
with respect to the nodes $\zeta'_i,$ $i=1,2,\ldots,s.$

Now define the function $f_N$ on $[-1,1]$ to piecewise consist of the polynomials 
$P_s(f,[t_{0,j}, t_{0,j+1}]),$ 
$P_s(f,[t_k,t_{k+1}]),$ $P_s(f,[\tau_{k+1},\tau_{k}]),$
$P_s(f,[\tau_{0,j+1},\tau_{0,j}]),$ 
 $j=0,1,\ldots,M-1,$ $k=1,2,\ldots,N-1.$

With these definitions in hand, we can now estimate the pointwise approximation error
$\|f-f_N\|.$ On the intervals  $\Delta^1_k = [t_k, t_{k+1}],$ $k=1,2,\ldots,N-1$ 
we obtain 
\begin{equation}
\|f-f_N\|_{C(\Delta_k^1)}\le  E_{s-1}(f,\Delta_k^1) (1+\lambda_s)\leq c\|f-T_{s-1}(f,\Delta_k^1,t_k)\|\lambda_s \le 
\frac{c (t_{k+1}-t_k)^s}{(k/N)^{v\gamma}s!} = cN^{-s},
\label{(2.2)}
\end{equation}
where  $\lambda_s$ is the Lebesgue constant with respect to the nodes used for 
$P_s(f,\Delta_{k}^1),$  $E_s(f,[a,b])$ is the best approximation of
the function $f$ by polynomials of degree at most $s$ in the norm of $C[a,b],$

Similar estimates are true in  $\Delta^2_k=[\tau_{k+1},\tau_k],$
$k=1,2,\ldots,N-1.$

Throughout this paper, we denote the constants that do not depend on $N$ by $c$.

Now let $k=0.$ 
It is well-known that it holds
$\|f-f_N\|_{C(\Delta_{0,0}^1)} \leq  E_{s-1}(f,\Delta_{0,0}^1)(1+\lambda_s),
$\\ 
where $\Delta_{0,0}^1=[t_{0,0}, t_{0,1}].$
Here  $\lambda_s$ is the Lebesgue constant with respect to the nodes used for 
$P_s(f,\Delta_{0,0}^1).$ 

Using Taylor's expansion  $T_{r-1}(f,\Delta^1_{0,0},-1)$ with the remainder in 
integral form we find
\begin{equation*}
E_{s-1}(f,\Delta^1_{0,0}) \leq \|f - T_{r-1}(f,\Delta^1_{0,0},-1)\|_{C(\Delta_{0,0})} 
\leq \frac{1}{(r-1)!}
\max\limits_{t\in\Delta^1_{00}}|\int\limits^t_{-1}f^{(r)}(\tau)(t-\tau)^{r-1}d\tau| \leq
\end{equation*}
\begin{equation*}
\leq \frac{1}{(r-1)!}
\max\limits_{t\in\Delta^1_{00}} \int\limits^t_{-1}(1+|\ln(1+\tau)|)(t-\tau)^{r-1}d\tau \leq
\frac{c}{r!}(h^r_{0,0}|\ln h_{00}| + h^r_{00})\leq c\frac{1}{N^s \ln^{r-1}N},
\end{equation*}
where $h_{0,k}=|t_{0,k+1} - t_{0,k}|,$ $k=0,1,\ldots,M-1.$

Therefore,
$
E_{s-1}(f,\Delta^1_{0,0}) \leq c N^{-s} \ln^{r-1}N.
$

Since the Lebesgue constant  $\lambda_s$ is independent of $N$ due to its invariance 
under rescaling we finally obtain
\begin{equation}
\|f -f_N\|_{C(\Delta_{0,0}^1)} \leq
c/(N^s \ln^{r-1}N).
\label{(2.3)}
\end{equation}

One can obtain a similar estimate in $\Delta^2_{0,0}=[\tau_{0,1},\tau_{0,0}].$

With similar arguments we obtain on   $\Delta_{0,j}^1=[t_{0,j},t_{0,j+1}], \ 1 \leq j \leq M-1$ 
\begin{equation}
\|f -f_N\|_{C(\Delta_{0,j}^1)} \leq
\frac{c \lambda_s}{(1+t_{0,j})^\gamma}h^s_{0,j} 
\leq \frac{c \lambda_{s}}{j^\gamma}(N^v \ln N)^\gamma
\left(\frac{1}{N^v \ln N}\right)^s\le  \frac{c}{N^s \ln^r N}.
\label{(2.4)}
\end{equation} 

Similar estimates are true in
$\Delta^2_{0j}=[\tau_{0,j+1},\tau_{0,j}],$ $j=1,2,\ldots,M-1.$

The total number $n$ of nodes used in constructing the local spline equals
$n=2(([ln N]-1)(s-1)+(N-1)(s-1))+1.$ Therefore $N\sim n/(2(s-1)).$  
 
It follows from (\ref{(2.2)}) -- (\ref{(2.4)}) that we have constructed a continuous local spline
$f_N$ which approximates $f \in \bar Q^1_{r\gamma}(\Omega,1)$ with the
accuracy $O(n^{-s}).$
From Definition \ref{Definition 1.1}, it follows that
\begin{equation}
d_n(\bar Q^1_{r\gamma}(\Omega,1),C) \leq cn^{-s}.
\label{(2.5)}
\end{equation}

Using Lemma \ref{Lemma 1.2} and the inequalities (\ref{(2.1)}), (\ref{(2.5)}) we complete the proof of the
theorem for $u=1.$
We now consider 

{ii)} $u \geq 2.$
We use the same points $t_k$ and $\tau_k$ and intervals $\Delta_k^1$ and $\Delta_k^2$ as before.
Additionally put  $M_0=\lceil\ln^{u/r}N\rceil,$ and  $M_k=\lceil\ln^{(u-1)/s}(N/k)\rceil,$
$k=1,\ldots,N-1.$ 
Divide each  $\Delta_k^1$ and $\Delta_k^2$  into $M_k,$ $k=0,1,\ldots,N-1,$ equal subintervals
and denote the latter ones by $\Delta^i_{k,j},$ $i=1,2,$
$j=0,1,\ldots,M_k-1,$ $k=0,1,\ldots,N-1.$

We shall approximate $f \in \bar Q^u_{r\gamma}(\Omega,1)$ within each $\Delta^i_{k,j}$
by the interpolating polynomial $P_s(f,\Delta^i_{k,j}),$ $i=1,2,$
$j=0,1,\ldots,M_k-1,$ $k=0,1,\ldots,N-1.$
Let $f_N$ be the spline composed of the polynomials $P_s(f,\Delta^i_{k,j}).$

Starting with approximating the error in $\Delta^i_0,$ we have in $\Delta^i_{0,0}$ 
$
\|f - f_N\|_{C(\Delta^1_{0,0})} \leq  E_{s-1}(f,\Delta^1_{0,0})(1+\lambda_s).
$

Using  Taylor's expansion $T_{r-1}(f, \Delta^1_{0,0}, -1)$ we find
\begin{equation*}
E_{s-1}(f,\Delta^1_{0,0}) \leq
\|f - T_{r-1}(f, \Delta^1_{0,0}, -1)\|_{C(\Delta_{0,0})} 
\leq \frac{1}{(r-1)!}\max\limits_{t \in \Delta^1_{0,0}}
\left|\int\limits^t_{-1} f^{(r)}(\tau)(t-\tau)^{r-1}d\tau\right| 
\end{equation*}
\begin{equation*}
\leq \frac{1}{(r-1)!}\max\limits_{t \in \Delta^1_{0,0}}
\left|\int\limits^t_{-1}(1+|\ln^u(1 + \tau)|)(t-\tau)^{r-1}d\tau\right| 
\leq ch^r_{00}|\ln^u h_{00}| \le
\end{equation*}
\begin{equation*}
\leq c\left(\frac{1}{N^v M_0}\right)^r
\left|\ln^u \left(\frac{1}{N^v M_0}\right)\right| \leq
 c\frac{1}{N^s \ln^u N}(\ln^u N + \ln^u \ln N) \leq c\frac{1}{N^s},
\end{equation*}
where $h_{00} = h_0/M_0,$ $h_0 = t_1 - t_0.$

Finally, 
$\|f - f_N\|_{C(\Delta^i_0)} \leq c N^{-s}$  in $\Delta^i_0, i=1,2.$

Our next step is to obtain the  estimate of $\|f - f_N\|_{C(\Delta^i_{kj})},$
 $j=0,1,\ldots,M_k-1,$ $i=1,2,$ $k=1,2,\ldots,N-1.$
It is obvious that
\begin{equation*}
\|f - f_N\|_{C(\Delta^1_{k,1})} \leq
\frac{c(t_{k,1} - t_{k,0})^s}{s!}
\left(\frac{N}{k}\right)^{v\gamma}\left(1+\left|\ln^{u-1}
\left(\frac{N}{k}\right)^{v\gamma}\right|\right) 
\end{equation*}
\begin{equation*}
\leq c\left(\frac{h_k}{M_k}\right)^s \left(\frac{N}{k}\right)^{v\gamma}
\left(1+\ln^{u-1}\frac{N}{k}\right) \leq
\end{equation*}
\begin{equation*}
\leq c\left(\left(\left(\frac{k+1}{N}\right)^v-\left(\frac{k}{N}\right)^v\right)
\frac{1}{\left(\ln \frac{N}{k}\right)^{(u-1)/s}}\right)^s
\left(\frac{N}{k}\right)^{v\gamma}\left(1+\ln^{u-1}\frac{N}{k}\right) 
\end{equation*}
\begin{equation*}
\leq c\frac{(k+\theta)^{(v-1)s-v\gamma}}{N^s} \leq \frac{c}{N^s},
\end{equation*}
where $t_{k,j}=t_k + (t_{k+1}-t_k)j/M_k, j=0,1,\cdots,M_k, k=0,1,\cdots,N-1.$

The errors 
$
\|f - f_N\|_{C(\Delta^1_{k,j})}, \quad \|f - f_N\|_{C(\Delta^2_{k,j})},  \quad
j=1,2,\ldots,M_k-1, \quad k=1,2,\ldots,N-1,
$
can be estimated similarly.

Combining the estimates obtained above we have
$
\|f - f_N\|_{C([-1,1])} \leq c N^{-s}.
$

Next we need to estimate the number of nodes  used to construct the local spline $f_N.$
For this purpose, we first estimate the number $m$ of subintervals $\Delta^i_{k,j},$ $i=1,2,$
$j=0,1,\ldots,M_k-1,$ $k=0,1,\ldots,N-1.$

Let $q=(u-1)/s.$ Then
\begin{equation*}
m=2\sum\limits^{N-1}_{k=0}M_k \leq
2\left(\ln^{\frac{u}{r}} N+ \sum\limits^{N-1}_{k=1} \ln^{\frac{u-1}{s}}
\frac{N}{k}\right) 
\leq c\left(N + \sum\limits^{N-1}_{k=2}
\ln^q \frac{N}{k}\right)
\end{equation*}
\begin{equation*}
\le c\left(N + N\int\limits^N_1 \frac{\ln^q t}{t^2} dt \right) \leq c N.
\end{equation*} 

Therefore the total number of nodes used in constructing 
$f_N$ equals to $n=(s-1)m+1=cN.$
Hence,
$
\|f - f_N\|_{C([-1,1])} \leq c N^{-s} \leq c n^{-s}.
$
 
Thus,
$
d_n(\bar Q_{r\gamma}^u(1),C) \le c n^{-s}.
$

Comparing the preceding inequality to the estimate (\ref{(2.1)}) we complete the proof of the theorem for
$u \geq 2.$
\end{proof}

\begin{theorem}
\label{Theorem 2.2}  \rm
Let $\Omega=[-1,1]$. Let $r,u$ be positive integers,  and $\gamma$ be a positive non-integer.
The estimate 
$
\delta_n(Q^u_{r\gamma}(\Omega,1))\asymp d_n(Q^u_{r\gamma}(\Omega,1), C)\asymp n^{-s}
$
holds, where as usual $s=r+\lceil\gamma\rceil.$
\end{theorem}

\begin{proof} First, we  estimate the infimum for
$\delta_n(Q^u_{r\gamma}(\Omega,1)).$  Note
$Q_{r\gamma}(\Omega,1) \subset \\ Q^u_{r\gamma}(\Omega,1). $ By Theorem \ref{Theorem 1.1},
the inequality $\delta_n(Q_{r\gamma}(\Omega,1))\geq cn^{-s}$ holds. Therefore
$\delta_n(Q^u_{r\gamma}(\Omega,1))\geq cn^{-s}.$ 
Next, we construct a local spline $f_N$
with at most $cn$ nodes approximating the given function $f \in Q^u_{r\gamma}(\Omega,1)$
with  error at most $cn^{-s}.$

We use the same subdivision procedure into intervals $\Delta^1_{k,j} = [t_{k,j}, t_{k,j+1}]$ 
and $\Delta^2_{k,j} = [\tau_{k,j+1}, \tau_{k,j}]$ as in the proof of Theorem \ref{Theorem 2.1}.
Here $M_0=\lceil\ln^{u/(r+1-\mu)} N\rceil, M_k=\lceil\ln ^{u/s}N\rceil, k=1,\cdots,N-1, \mu=1-\zeta.$

We approximate a function $f$ on $[-1,1]$ by the spline  $f_N$ which
piecewise consists of the polynomials
 $P_s(f,\Delta_{k,j}^i),$
$i=1,2,$ $j=0,1,\cdots,M_k-1,$  $k=0,1,\ldots,N-1.$

Estimating the error $\|f-f_N\|$ we obtain  for  $j=0,\cdots,M_k-1,  1\leq k\leq N-1$
\begin{equation*}
||f-f_N||_{C(\Delta^1_{k,j})}\leq \frac{c (t_{k+1}-t_k)^s}{M_k^ss!}\left( \frac{N}{k} \right)^{v\gamma}
\left(1+\left|\ln^u\left(\frac{k}{N} \right)^{v}\right|\right )\leq
\end{equation*}
\begin{equation*}
\leq c\left(\left(\left( \frac{k+1}{N}\right )^v- \left( \frac{k}{N}\right )^v\right)
\frac{1}{M_k}\right)^s\left( \frac{N}{k} \right)^{v\gamma}\left(1+ \left|\ln^u\left(\frac{N}{k} \right)^{v}\right|  \right)\leq\frac{c}{N^s}.
\end{equation*} 

 One can derive similarly
$ ||f-f_N||_{C(\Delta^2_{k,j})}\leq c/N^s,$
$k=1,2,\ldots,N-1. $

For $k=0$, the following estimate is true:
$ ||f-f_N||_{C(\Delta^1_{0,0})}\leq E_{s-1}(f,\Delta^1_{0,0})(1+\lambda_{s}).$

Using the Taylor expansion $T_r(f,\Delta^1_{0,0},-1)$ we  have
\begin{equation*}
E_{s-1}(f,\Delta^1_{0,0})\leq||f-T_r(f,\Delta^1_{0,0},-1)||_{C(\Delta^1_{0,0})} \leq
\end{equation*}
\begin{equation*}
\leq \frac{1}{r!}\max_{t\in \Delta^1_{0,0}}\int\limits_{-1}^t \frac{(1+|\ln^u(1+\tau)|)}
{(1+\tau)^\mu}(t-\tau)^r d\tau 
\leq ch^{r+1-\mu}_{00}|\ln^uh_{00}|\leq cN^{-s}.
\end{equation*} 
Recall $h_{00}=N^{-v}/M_0\le c N^{-v} (ln N)^{-u/(r+1-\mu)}.$
Hence, $||f-f_N||_{C(\Delta^1_{0,0})} \leq cN^{-s}.$ 

One can estimate the norms  $||f-f_N||_{C(\Delta^1_{0,j})},$ 
and $||f-f_N||_{C(\Delta^2_{0,j})},$ \ \ $j=0,1,\ldots,M_0-1$ in a similar way.

Thus, we have obtained 
\begin{equation}
||f-f_N||_{C(\Omega)}\leq cN^{-s}.
\label{(2.6)}
\end{equation}

It remains to estimate the number  $n$ of nodes of the local spline $f_N.$
Repeating the arguments used
in the proof of the preceding theorem we derive $n \asymp N$ since the specific value of $q$
did not enter in this calculation.
Thus, we have 
$\|f-f_N\|_{C(\Omega)}\leq cn^{-s}$ and 
$d_n(Q^r_{r\gamma}(\Omega,1),C) \le cn^{-s}.$
Comparing the last inequality with the estimate
$\delta_n(Q^r_{r\gamma}(\Omega,1)) \ge cn^{-s}$
we complete the proof of Theorem.
\end{proof}

\section{Widths of the classes  $\bar Q^u_{r \gamma}(\Omega,1)$ and
$Q^u_{r \gamma}(\Omega),1), \Omega = [-1,1]^l.$ }
\label{sec;Sec3}

In this section we estimate the Kolmogorov and Babenko widths 
for each of the functional classes
$\bar Q^u_{r \gamma}(\Omega,1)$ and $Q^u_{r \gamma}(\Omega,1),$
$\Omega =[-1,1]^l,$ $l=2,3,\ldots.$

\begin{theorem} \rm
\label{Theorem 3.1}  Let $\Omega = [-1,1]^l,$ $l \geq 2,$
$ u=1,2,\cdots,$ $v=s/(s-\gamma).$
The following estimates hold
\begin{equation}
\delta_n(\bar Q^u_{r \gamma}(\Omega,1)) \asymp d_n(\bar Q^u_{r \gamma}(\Omega,1),C) \asymp
n^{-s/l} 
\label{(3.1)}
\end{equation}
$ {\rm if} \quad   v<l/(l-1),$
\begin{equation}
c n^{-s/l}(\ln n)^{u-1+s/l} \le \delta_n(\bar Q^u_{r \gamma}(\Omega,1))
 \le 2d_n(\bar Q^u_{r \gamma}(\Omega,1),C) \leq
\label{(3.2)}
\end{equation}
\begin{equation}
\leq c
\begin{cases}
n^{-s/l}(\ln n)^{us/r},  & u/r \geq 1/l+(u-1)/s, \\ 
n^{-s/l}(\ln n)^{u-1+s/l},  & u/r \leq 1/l+(u-1)/s
\end{cases}
\label{(3.3)}
\end{equation}
$
{\rm if} \quad  v=l/(l-1).
$
\end{theorem}

\begin{proof}
We start to estimate the Kolmogorov widths.
First, we construct a local spline not necessarily continuous which approximates the functions 
of the classes $\bar Q^u_{r \gamma}(\Omega,1)$ for $v\le l/(l-1)$  and
has  the error  given in the right-hand side of (\ref{(3.1)}) $-$ (\ref{(3.3)}). 
Afterwards we construct a  continuous local spline having the same error of approximation. 
This requires some modifications which we indicate below.

Let $\Delta^k$ denote the set 
\begin{equation*}
\Delta^k =\left\{ t \in \Omega:\left( \frac{k}{N} \right)^v  \leq d(t,\Gamma) \leq 
\left( \frac{k+1}{N} \right)^{v}, \ k=0,1,\ldots,N-1, \right\}
\end{equation*}
where $d(t,\Gamma)$ is as in Definition \ref{Definition 1.3}.

We now partition the domains $\Delta^k,$ $k=0,\ldots,N-1,$ in the following way. 
Decompose each $\Delta^k$ into cubes and parallelepipeds $\Delta^k_{i_1,\ldots,i_l}$
with their edges parallel to the axes.
The lengths of edges are not less than the value $h_k$ and less than $2h_k,$
$h_k=((k+1)/N)^v-(k/N)^v, k=0,1,\dots,N-1.$ (See Fig \ref{fig:1})

\begin{figure}[t!]
\centering
\includegraphics[width=5.0in]{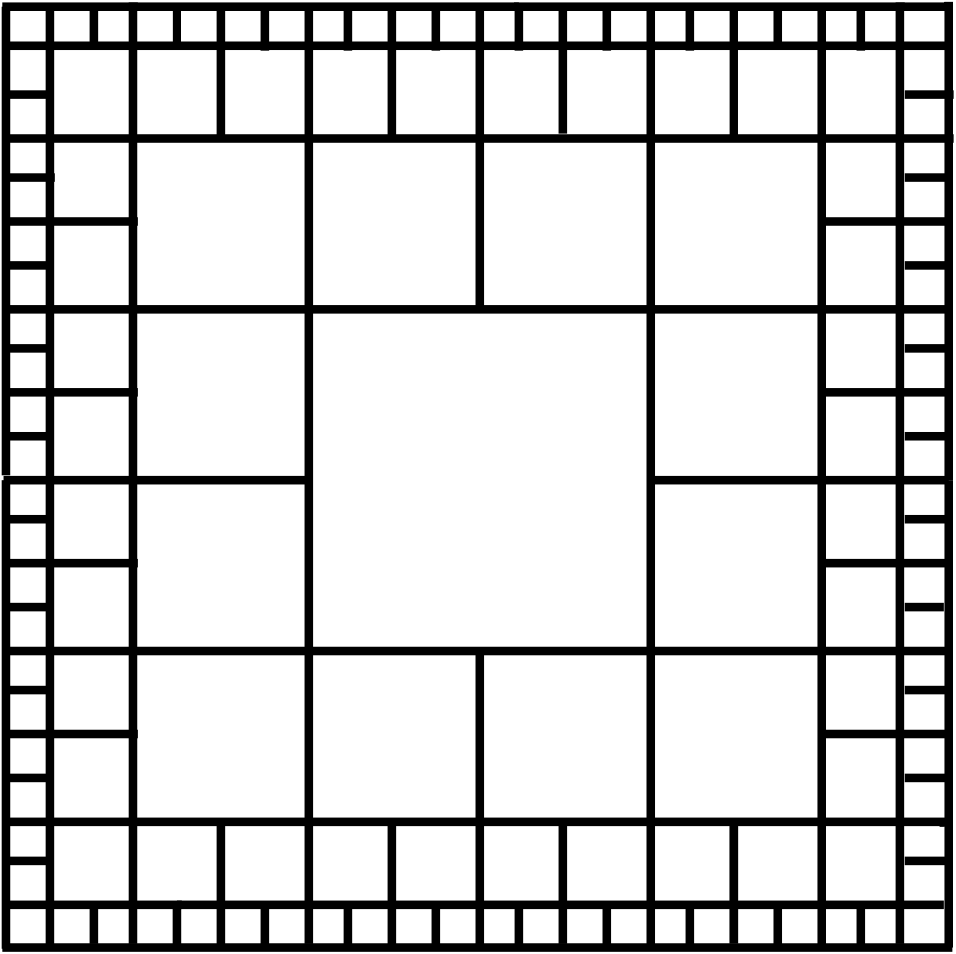}
\caption{Partition of the domain $\Omega=[-1,1]^2$ by subdomains $\Delta^k_{i_1,i_2}$.}
\label{fig:1}
\end{figure}

Let estimate the number of $\Delta^k_{i_1,\dots,i_l},$ $k=0,1,\ldots,N-1.$ Clearly,
\begin{equation*}
n \ge 1+m \sum\limits_{k=0}^{N-1} \left[\frac{2-2((k+1)/N)^v}{2 h_k}\right]^{l-1} = 
1+m \sum\limits_{k=0}^{N-1} \left[\frac{N^v-(k+1)^v}{(k+1)^v-k^v}\right]^{l-1} \ge
\end{equation*}
\begin{equation}
\ge c  
\begin{cases}
N^{v(l-1)}, &   v>l/(l-1),\\
N^l, & v<l/(l-1), \\
N^l ln N, & v=l/(l-1),
\end{cases}
\label{(3.4)}
\end{equation}
where $m$  is the number of faces in $\Omega.$ 

Similarly,
\begin{equation}
n \le 1+m \sum\limits_{k=0}^{N-1} \left(\frac{2-2(k/N)^v}{h_k} +1\right)^{l-1} 
\le c  
\begin{cases}
N^{v(l-1)}, &   v>l/(l-1),\\
N^l, & v<l/(l-1), \\
N^l ln N, & v=l/(l-1),
\end{cases}
\label{(3.5)}
\end{equation}

Thus,
\begin{equation*}
n \asymp  
\begin{cases}
N^{v(l-1)}, &   v>l/(l-1),\\
N^l, & v<l/(l-1), \\
N^l ln N, & v=l/(l-1),
\end{cases}
\end{equation*}

Let
$
M_0=\lceil (\ln N)^{u/r}\rceil, \ \ M_k=\lceil (\ln N/k)^{(u-1)/s}\rceil, \ \ k=1,2,\ldots, N-1.
$
Dividing each edge of $\Delta^k_{i_1,\ldots,i_l}$ into $M_k$ equal subintervals and
passing  the planes parallel to the coordinate planes  throught the points of division
we partition $\Delta^k_{i_1,\cdots,i_l}$ into 
$\Delta^k_{i_1,\ldots,i_l;j_1,\ldots,j_l}.$

In Section \ref{sec;Sec2} we used the interpolating polynomial $P_s(f,[a,b])$ for functions
$f$ of one variable. As a next step we consider a possible multivariate counterpart.
More precisely, for a function $f(t_1,\ldots,t_l)$ of $l$ variables on $[a_1,b_1;\cdots;a_l,b_l]$
we define the interpolating polynomial $P_{s,\ldots,s}(f,[a_1,b_1;\ldots;a_l,b_l])$ iteratively:
$
P_{s,\ldots,s}(f,[a_1,b_1;\ldots;a_l,b_l])= 
P^{t_1}_s(P^{t_2}_s(\cdots P^{t_l}_s(f;[a_l,b_l]);[a_{l-1},b_{l-1}]); \cdots;
[a_1,b_1]).
$
This polynomial then is of degree $s-1$ in each of the variables $t_1,\ldots,t_l.$ In other words,
$P^{t_l}_s(f;[a_l,b_l])$ interpolates  $ f(t_1,\ldots,t_l)$ with respect to $t_l \in [a_l,b_l];$
$P^{t_{l-1}}_{s}(P^{t_l}_s(f;[a_l,b_l]); [a_{l-1},b_{l-1}])$
interpolates  $P_s^{t_l}(f,[a_l,b_l])$ in
$t_{l-1} \in [a_{l-1}, b_{l-1}]$ etc.
The polynomial $P_{s,\ldots,s}(f,\Delta^k_{i_1,\ldots,i_l;j_1,\ldots,j_l})$ interpolates
$f(t_1,\ldots,t_l)$ in each  $\Delta^k_{i_1,\ldots,i_l;j_1,\ldots,j_l}.$ We piece together
the interpolating polynomials   $P_{s,\ldots,s}(f,\Delta^k_{i_1,\ldots,i_l;j_1,\ldots,j_l})$ 
and construct a local spline $f_N.$
Next we estimate the approximation of
$f\in \bar Q^u_{r\gamma}(\Omega,1)$  by $f_N.$

Let $k=0.$ Then,
$
\|f-P_{s,\ldots,s}(f,\Delta^0_{i_1,\ldots,i_l;j_1,\ldots,j_l})\|\leq
c E_{r-1,\ldots,r-1}(f,\Delta^0_{i_1,\ldots,i_l;j_1,\ldots,j_l})\lambda_r^l.
$
where $E_{r,\ldots,r}(f; \Delta^0_{i_1,\ldots,i_l; j_1,\ldots,j_l})$ is
the best approximation to a function $f$ in the space $C$ by a  polynomial of 
degree $r$ in each variable in $\Delta^0_{i_1,\ldots,i_l; j_1,\ldots,
j_l}.$

To estimate $E_{r,\ldots,r}(f, \Delta^0_{i_1,\ldots,i_l; j_1,
\ldots,j_l}),$ we use Taylor's expansion with the remainder
in the integral form (see e.g.  \citep{[26]})
\begin{equation}
f(t_1,\ldots,t_l) =\sum\limits^r_{\alpha=0}
\frac{1}{\alpha!} \sum\limits^l_{j_1=1} \cdots \sum\limits^l_{j_\alpha=1}
(t_{j_1} - t^0_{j_1})\cdots (t_{j_\alpha} - t^0_{j_\alpha})
\frac{\partial^{\alpha} f(t^0)}{\partial t_{j_1} \cdots \partial t_{j_\alpha}}+
R_{r+1}(t),  
\label{(3.6)}
\end{equation}
where
\begin{equation*}
R_{r+1}(t) = \frac{1}{r!}\int\limits^1_0 (1-\tau)^r
\sum\limits^l_{j_1=1} \cdots \sum\limits^l_{j_{r+1}=1}
(t_{j_1} - t^0_{j_1}) \cdots
(t_{j_{r+1}} - t^0_{j_{r+1}})
\frac{\partial^{r+1} f(t^0 +\tau(t-t^0))}{\partial t_{j_1} \cdots \partial
t_{j_{r+1}}}d\tau=
\end{equation*} 
\begin{equation*}
= (r+1)\sum\limits_{|\alpha|=r+1} \frac{(t-t^0)^\alpha}{\alpha!} \int\limits^1_0 (1-\tau)^r
f^{(\alpha)}(t^0+\tau(t-t^0))d\tau.
\end{equation*}

With $t,t_0 $ in the domain $\Delta^0_{i_1,\ldots,i_l;j_1,\ldots,j_l}$, which has a
nonempty intersection with the boundary $\Gamma=\partial \Omega,$
 we trivially have 
$d(t^0+\tau(t-t^0),\Gamma) \le h_{00}=h_0/\lceil(ln N)^{u/r}\rceil,$ and thus
$|f^{(r)}(t^0+\tau(t-t^0)|\le 1+|ln^u h_{00}|,$ which immediately yields
\begin{equation*}
E_{r-1,\ldots,r-1}(f,\Delta^0_{i_1,\ldots,i_l;j_1,\ldots,j_l})\leq
 c h_{00}^r\int\limits_0^1(1-\tau)^{r-1}(1+|\ln^u(\tau h_{00})|)d\tau
\end{equation*} 
\begin{equation*}
\leq ch_{00}^r\ln^uh_{00}\leq c\left(\frac{1}{N}\right)^s,
\end{equation*}
where $h_{00}=h_0/M_0,\ \ h_0=1/N^v.$

Hence,

\begin{equation}
\|f-P_{s,\ldots,s}(f,\Delta^0_{i_1,\ldots,i_l;j_1,
\ldots,j_l})\|_{C(\Delta^0_{i_1,\ldots,i_l;j_1,\ldots,j_l})  }\leq cN^{-s}.
\label{(3.7)}
\end{equation}

The estimate is valid for all  $\Delta^0_{i_1,\ldots,i_l;j_1,\ldots,j_l}$.

Now let $1\leq k\leq N-1.$ Then,
\begin{equation*}
\|f-P_{s,\ldots,s}(f,\Delta^k_{i_1,\ldots,i_l;j_1,\ldots,j_l})\|_{C(\Delta^k_{i_1,\ldots,i_l;j_1,\ldots,j_l})  }\leq
\end{equation*}
\begin{equation}
\leq c\left(\left(\left(\frac{k+1}{N}\right)^v-\left(\frac{k}{N}\right)^v\right)\frac{1}
{(\ln \frac{N}{k})^{(u-1)/s}}\right)^s\frac{(1+|\ln (\frac{k}{N})^v|)^{u-1}}
{\left(\frac{k}{N}\right)^{v\gamma}}\leq \frac{c}{N^s}.
\label{(3.8)}
\end{equation}

Combining (\ref{(3.7)}) and (\ref{(3.8)}) we conclude
\begin{equation}
\|f-f_N\|\leq c N^{-s}.
\label{(3.9)}
\end{equation}

Estimating the number of nodes used in constructing  $f_N$
we study two cases i) $v<l/(l-1)$ and ii) $v=l/(l-1).$

{ i).} Let $v<l/(l-1)$. The upper estimate follows immediately from the chain of inequalities
\begin{equation*}
n\leq m \sum\limits^{N-1}_{k=1} \left(\frac{2-2(\frac{k}{N})^v}{(\frac{k+1}{N})^v - (\frac{k}{N})^v}\right)^{l-1} M_k^l+
mN^{v{(l-1)}} [\ln N]^{lu/r} \leq
\end{equation*} 
\begin{equation*}
\leq c N^{v(l-1)} (\ln N)^{lu/r} + c \sum \limits_{k=1}^{N-1}
 \left( \frac{2N^v-2k^v}{v(k+\theta)^{v-1}} \right)^{l-1}\left(1+\left(  \ln \frac{N}{k}\right)^{\frac{u-1}{s}} \right)^l\le
\end{equation*} 
\begin{equation}
\leq c N^{v(l-1)} (\ln N)^{lu/r}+c\sum \limits_{k=1}^{N-1}\frac{N^{v(l-1)}}{k^{(v-1)(l-1)}}
\left(\left( \ln \frac{N}{k} \right)^{\frac{(u-1)l}{s}}+1\right)\le c N^l,
\label{(3.10)}
\end{equation}
where $m$ is the number of faces of $\Omega.$

The inequalities (\ref{(3.9)}) and (\ref{(3.10)}) yield  
 $\|f-f^*_N\| \le c n^{-s/l}$ for $v<l/(l-1).$

{ ii).} Let $v=l/(l-1).$ Just as for $v<l/(l-1),$ the upper bound  follows from the
chain of inequalities
\begin{equation*}
n\leq m\sum\limits^{N-1}_{k=1}\left( \frac{2-2(\frac{k}{N})^v}{(\frac{k+1}{N})^v-
(\frac{k}{N})^v} \right)^{l-1}M_k^l+mN^{v(l-1)}[\ln N]^{lu/r}\leq
\end{equation*} 
\begin{equation*}
\leq c N^l (\ln N)^{lu/r} + c \sum\limits^{N-1}_{k=1}\left(\frac{N^v}{k^{v-1}} \right)^{l-1}
\left[\ln\frac{N}{k}\right]^{(u-1)l/s}\leq
\end{equation*} 
\begin{equation*}
\leq c N^l (\ln N)^{lu/r} + \frac{cN^l}{N^{(v-1)(l-1)}}\sum\limits^{N-1}_{k=1}\frac{N^{(v-1)(l-1)}}{k^{(v-1)(l-1)}}
\left(\ln\frac{N}{k}\right)^{(u-1)l/s}\leq
\end{equation*} 
\begin{equation*}
\leq c N^l (\ln N)^{lu/r} + c N^{l-1}\int\limits^N_1\frac{N}{x}\left(\ln\frac{N}{x}\right)^{(u-1)l/s}dx\le
\end{equation*} 
\begin{equation}
\le c
\begin{cases}
N^l(\ln N)^{lu/r}, & lu/r \geq 1+(u-1)l/s, \\
N^l(\ln N)^{(u-1)l/s+1}, & lu/r \leq 1+(u-1)l/s.
\end{cases}
\label{(3.11)}
\end{equation}

From (\ref{(3.9)}), (\ref{(3.11)}), and for $v=l/(l-1),$ it follows
\begin{equation*}
\|f - f_N\| \leq c 
\begin{cases}
\left(\frac{1}{n}\right)^{s/l} (\ln n)^{us/r}, & u/r \geq 1/l+(u-1)/s, \\
\left(\frac{1}{n}\right)^{s/l}(\ln n)^{u-1+s/l},  & u/r \leq 1/l+(u-1)/s.
\end{cases}
\end{equation*} 
where $n$ is the number of nodes of the local spline.

Now we describe the necessary modifications to the computations above in order  to
construct a continuous local spline 
approximating $\bar Q^u_{r \gamma}(\Omega,1)$ for $v\le l/(l-1)$  and
having the estimate given in (\ref{(3.1)}) - (\ref{(3.3)}). 
Let $\Delta^k$  be defined as above and $h_k=((k+1)/N)^v -(k/N)^v,$ $h_k^*= h_k/M_k,$ 
 $k=0,1,\ldots,N-1.$ We then shall use
the following modified partitioning of  $\Omega.$  As before decompose $\Delta^{N-2}$
into  cubes or parallelepipeds $\Delta^{N-2}_{i_1,\ldots,i_l}$ with edges parallel to the axes,
and whose lengths is not less then  $h_{N-2}$ and does not exceed then $2h_{N-2},$ 
but now choose the partition in such a way that
the vertices of $\Delta^{N-1}$ are contained in the set consisting of all vertices
of the subdomains $\Delta^{N-2}_{i_1,\ldots,i_l},$

Next, partition each $\Delta^{N-2}_{i_1,\ldots,i_l}$
into  cubes or parallelepipeds $\Delta^{N-2}_{i_1,\ldots,i_l;j_1,\ldots,j_r}$ with edges parallel to the axes.
We divide each edge of the  subdomain $\Delta^{N-2}_{i_1,\ldots,i_l}$ 
into $M_{N-2}$ equal subintervals and pass the planes parallel to the coordinate planes   through the points of division.
Finally we decomposed  $\Delta^{N-2}$ into  $\Delta^{N-2}_{i_1,\ldots,i_l;j_1,\cdots,j_l}.$

To decompose $\Delta^{N-3}$ into cubes and parallelepipeds 
$\Delta^{N-3}_{i_1,\ldots,i_l;j_1,\ldots,j_l}$ we pass the planes parallel to the coordinate planes 
through the vertices of  $\Delta^{N-2}_{i_1,\ldots,i_l;j_1,\ldots,j_l},$ which are located 
on a common face  of the hyperplane $\Delta^{N-2} \cap \Delta^{N-3}.$  Denote
the obtained subdomains by  $g^{N-3}_{i_1,\ldots,i_l;j_1,\ldots,j_l}.$
Let $h^*_{N-3}=h_{N-3}/M_{N-3}.$ Consider $g^{N-3}_{i_1,\ldots,i_l;j_1,\ldots,j_l}=[a_1,b_1;\cdots;a_l,b_l].$
If the length of the edge $(a_k,b_k)$ exceeds $2h^*_{N-3},$ 
we divide $(a_k,b_k)$ into $[|b_k-a_k|/h^*_{N-3}]$ equal subintervals
and pass the planes parallel to the coordinate planes through the points of division. 
We shall refer to the result of this procedure as $\Delta^{N-3}_{i_1,\ldots,i_l;j_1,\ldots,j_l}$.
This way we have  $\Delta^{N-3}$ decomposed into  $\Delta^{N-3}_{i_1,\ldots,i_l;j_1,\ldots,j_l}.$ 
Continuing this process 
we partition the  domain $\Omega$  into subdomains $\Delta^k_{i_1,\ldots,i_l;j_1,\ldots,j_l}, 
k=0,1,,\ldots,N-2.$

One estimates the total number of $\Delta^{k}_{i_1,\ldots,i_l;j_1,\cdots,j_l}$
in $\Omega$ using (\ref{(3.10)}), (\ref{(3.11)}).

Now we construct the continuous spline $f_N$ approximating 
a function $f$ of $l$ variables.
The polynomial $P_{s,\ldots,s}(f,\Delta^{N-1})$
interpolates  $f$ in $\Delta^{N-1};$ $P_{s,\ldots,s}(\tilde f,\Delta^{N-2}_{i_1,\ldots,i_l;j_1,\cdots,j_l})$
interpolates $\tilde f$  in $\Delta^{N-2}_{i_1,\ldots,i_l;j_1,\cdots,j_l}.$ 
We say that the function $\tilde f$ equals to $f$ at all  points of interpolation except for those
located on the hypersurface $\Delta^{N-1} \cap \Delta^{N-2}$. At those points, 
$\tilde f$ equals to $P_{s,\ldots,s}( f,\Delta^{N-1}).$ Continuing this process, we construct the interpolating polynomials 
 $P_{s,\ldots,s}(\tilde f,\Delta^{k}_{i_1,\ldots,i_l;j_1,\cdots,j_l}), k=N-3, \cdots,1,0.$

Next we piece together all the interpolating polynomials $P_{s,\ldots,s}(f,\Delta^{N-1}),$\\
$P_{s,\ldots,s}(\tilde f,\Delta^k_{i_1,\ldots,i_l;j_1,\cdots,j_l}),$
$k=0,1,2,\ldots,N-2,$ that interpolate $f$ within each  $\Delta^{N-1},$ 
$\Delta^k_{i_1,\ldots,i_l;j_1,\cdots,j_l},$ and construct 
the continuous local spline $f_N^*.$

Repeating the above computations for a non-continuous local spline we obtain\\
$
\|f - f_N^*\|_{C(\Delta^k_{i_1,\ldots,i_l;j_1,\cdots,j_l})} \leq cN^{-s}, k=0,1,2,\cdots, N-2,
$
$
\|f - f_N^*\|_{C(\Delta^{N-1})} \leq cN^{-s}.
$

Therefore
\begin{equation}
\|f - f_N^*\|_{C(\Omega)} \leq cN^{-s}.
\label{(3.12)}
\end{equation}

Using the  inequalities  (\ref{(3.10)}), (\ref{(3.11)}), (\ref{(3.12)})  we have proved the following statements

\begin{equation}
d_n(\bar Q^u_{r \gamma} (\Omega,1), C) \leq c n^{-s/l}
\label{(3.13)}
\end{equation}
if $v < l/(l-1),$ 

\begin{equation}
d_n(\bar Q^u_{r \gamma} (\Omega, 1), C) \leq c \left \{
\begin{array}{cc}
n^{-s/l} (\ln n)^{u s/r}, \, u/r \geq 1/l + (u-1)/s,\\
n^{-s/l} (\ln n)^{u-1 + s/l}, \, u/r \leq 1/l + (u-1)/s\\
\end{array}
\right.
\label{(3.14)}
\end{equation}
if $v = l/(l-1).$

Let estimate $\delta_n(\bar Q_{r\gamma}^u(\Omega,1))$ for $v=s/(s-\gamma), v<l/(l-1).$

Note that for every given positive integer $u$ we have 
 $Q_{r \gamma}(\Omega,1) \subset \bar Q^u_{r \gamma}(\Omega,1),$
which together with Theorem \ref{Theorem 1.3} yields
\begin{equation}
\delta_n(\bar Q^u_{r \gamma}(\Omega,1)) \geq
\delta_n(Q_{r\gamma}(\Omega,1)) \asymp
n^{-s/l},  v < l/(l-1).
\label{(3.15)}
\end{equation}

Let estimate $\delta_n(\bar Q_{r\gamma}^u(\Omega,1))$ for $v=s/(s-\gamma), v=l/(l-1).$

We decompose the domain $\Omega$ into subdomains
$\Delta^k_{i_1,\ldots,i_l},$ $k=0,1,\ldots,N-1$
following the procedure which was described above in the proof of Theorem (see the part
of constructing  a not necessarily continuous local spline).

Let
$$
M_k = \left \{
\begin{array}{cc}
\lceil(\ln N)^{(u-1)/s}\rceil, \, k=0\\
\lceil\left(\ln \frac{N}{k}\right)^{(u-1)/s}\rceil, \, k = 1,2,\ldots,N-1.\\
\end{array}
\right.
$$

We divide each edge of $\Delta^k_{i_1,\ldots,i_l},$ $k=0,1,\ldots,N-1$ into $M_k$ equal subintervals
and pass the planes parallel to
coordinate planes through the points of division. This way we have
$\Delta^k_{i_1,\ldots,i_l},$ $k=0,1,\ldots,N-1$ decomposed into
$\Delta^k_{i_1,\ldots,i_l;j_1,\cdots,j_l},$ $k=0,1,\ldots,N-1.$

Let estimate the number   $\Delta^k_{i_1,\ldots,i_l; j_1,\ldots,j_l},$  $k=0,1,\cdots,N-1.$ Clearly
$$
n \asymp m\sum\limits^{N-1}_{k=1}
\left(\frac{2-2\left(\frac{k}{N}\right)^v}{\left(\frac{k+1}{N}\right)^v - \left(\frac{k}{N}\right)^v}\right)^{l-1} M^l_k + m N^{v(l-1)}(\ln N)^{l(u-1)/s} \asymp
$$
$$
\asymp N^l(\ln N)^{l(u-1)/s} + N^{l-1}\int\limits^N_1
\frac{N}{x}\left(\ln \frac{N}{x}\right)^{(u-1)l/s}dx \asymp N^l (\ln N)^{(u-1)l/s + 1}.
$$

Let $\Delta^k_{i_1,\ldots,i_l; j_1,\ldots,j_l} = [b_{i_1,j_1}, b_{i_1, j_1+1}; \cdots, b_{i_l,j_l}, b_{i_l, j_l+1}].$
Introduce the functions
\begin{equation*}
\varphi^k _{i_1,\ldots,i_l; j_1,\ldots,j_l}(t)  = 
\end{equation*}
\begin{equation}
=\left \{
\begin{array}{cc}
A_k \frac{((t_1 - b_{i_1,j_1})(b_{i_1,j_1+1} - t_1) \cdots (t_l - b_{i_l,j_l})(b_{i_l,j_l+1} - t_l))^s}{(h_k/M_k)^{(2l-1)s}
 ((k+1)/N)^{v\gamma}} \left(1+\left|\ln^{u-1}\left(\frac{k+1}{N}\right)^v\right|\right),\\  
t\in \Delta^k_{i_1,\ldots,i_l; j_1,\ldots,j_l},
0, \ \ \ t\in \Omega\setminus \Delta^k_{i_1,\ldots,i_l; j_1,\ldots,j_l},
\end{array}
\right.
\label{(3.16)}
\end{equation}
$k=0,1,\ldots,N-2.$

Let $\Delta^{N-1} = [b_{i_1}, b_{i_1+1}; \cdots, b_{i_l}, b_{i_l+1}].$
Introduce the function
\begin{equation}
\varphi^{N-1}(t)  =
\left \{
\begin{array}{cc}
A_{N-1} \frac{((t_1 - b_{i_1})(b_{i_1+1} - t_1) \cdots (t_l - b_{i_l})(b_{i_l+1} - t_l))^s}{h_{N-1}^{(2l-1)s}} ,\ \ \ 
t\in \Delta^{N-1},\\
0, \ \ \ t\in \Omega\setminus \Delta^{N-1},
\end{array}
\right.
\label{(3.17)}
\end{equation}

Constants $A_k,$ $k=0,1,\ldots,N-1,$ are chosen such that
\begin{equation*}
|D^s\varphi^k_{i_1,\cdots,i_l}|
\le \left(1+\left|\ln^{u-1}\left(\frac{k+1}{N}\right)^v\right|\right)/\left(\frac{k+1}{N}\right)^{v(s-r)}.
\end{equation*}

Constant $A_{N-1}$ is chosen such that
$
|D^s\varphi^{N-1}|
\le 1.$

Let $\xi(t)$ be a linear combination
$
\xi(t) = \sum\limits_{k, i_1,\ldots,i_l;j_1,\cdots,j_l}
 C^k_{i_1,\ldots,i_l;j_1,\cdots,j_l} \varphi^k_{i_1,\ldots,i_l;j_1,\cdots,j_l}(t),
$
where $|C^k_{i_1,\ldots,i_l;j_1,\cdots,j_l}| \leq 1.$
Here the summation is taken over all domains $\Delta^k_{{i_1,\ldots,i_l;j_1\cdots,j_l}}$ of  $\Omega.$

Repeating the arguments presented in \cite{[2],[3],[24]} we have 
$
\delta_n(\bar Q^u_{r \gamma}(\Omega, 1)) \geq cN^{-s} \geq c n^{-s/l}(\ln n)^{u-1+s/l}.
$

From the inequalities (\ref{(3.13)}) -- (\ref{(3.14)}) and estimates of the Babenko widths 
we complete the proof of Theorem.

\end{proof}

\begin{rmk}
Let $v=l/(l-1).$
The estimate 
$
d_n(\bar Q^u_{r\gamma}(\Omega,1),C) \geq c n^{-s/l}(\ln n)^{u-1+s/l}
$
follows from the definition of $\xi(t)$ and Lemma \ref{Lemma 1.1}.
\end{rmk}

For $v > l/(l-1),$ we state the following

\begin{theorem} \rm
\label{Theorem 3.2} Let $\Omega = [-1, 1]^l,$ $l \geq 2,$ $u=1,2,\cdots, $ $v = s/(s-\gamma),$ $v > l/(l-1).$ The estimate
$
\delta_n(\bar Q^u_{r\gamma}(\Omega, 1)) \geq c n^{-(s-\gamma)/(l-1)} \ln^{u-1} N
$
holds.
\end{theorem}

\begin{proof}
Let   $\Delta^0$  be the set $
\Delta^0=\left\{t \in \Omega: 0  \leq d(t,\Gamma) \leq  (1/N^v)= \rho_0 \right\}.$

Let $\Delta^k$ be the set $
\Delta^k=\left\{t \in \Omega:\rho_{k-1} \leq d(t,\Gamma) \leq \rho_k \leq 1\right\},$
where $\rho_k$ is defined by
$
h_k^s/\rho_k^\gamma = N^{-s} \ln^{u-1}N,$
and $h_k = \rho_k - \rho_{k-1},$ $k=1,2,\ldots,m.$
Here $m$ is the largest integer value when $\rho_m \leq 1.$ 
If $\rho_m =1,$ then  $\Omega$ is decomposed into $\Delta^k,$ $k=0,1,\ldots,m.$ 
If $\rho_m < 1,$ then  $\Delta^{m+1}$ is the set 
$
\Delta^{m+1}=\left\{t \in \Omega:\rho_{m} \leq d(t,\Gamma) \leq 1\right\}.
$

Without loss of generality we demonstrate our computations for $\rho_m = 1.$
Now we show that the equations
$
h_k^s/\rho_k^\gamma = N^{-s} \ln^{u-1} N
$
are solvable.

Let $\rho_k^*=(k/N)^v, k=0,1,\cdots,N, h_k^*=\rho_{k}^*-\rho_{k-1}^*, k=1,\cdots,N.$ 
Then
$h_1^{*s}/\rho_1^{*\gamma}=1/N^s$ if $k=1.$ 
For  $k=2,\ldots,N,$ we have
\begin{equation*}
\frac{h_k^{*s}}{\rho_k^{*\gamma}} = \frac{(k^v - (k-1)^v)^s}{(k/N)^{v\gamma}N^{vs}} = 
\frac{(v(k-\Theta)^{v-1})^s}{k^{v\gamma}} \frac{1}{N^s} 
\geq \left(\frac{k-1}{k}\right)^{v\gamma} v^s  \frac{1}{N^s}
\geq \left(\frac{1}{2}\right)^{v\gamma} v^s \frac{1}{N^s}.
\end{equation*}

Thus there exists a sequence
$\rho_k^* = (k/N)^{v},$ $k=0,1,\ldots,N,$ such that\\
$
h_k^{*s}/\rho_k^{*\gamma} \geq (1/2)^{v\gamma}(v/N)^s  
 = cN^{-s}, \, h_k^* = \rho_k^* - \rho_{k-1}^*. 
$

On the other hand, $\varphi(\rho) = (\rho - \rho_{k-1})^s/\rho^\gamma$ 
is an increasing function if $\rho > \rho_{k-1}$ for any $\rho_{k-1}.$ 

Therefore there exists a sequence $\rho_k$
such that  $(\rho_k - \rho_{k-1})^s/(\rho_k)^\gamma \geq N^{-s} \ln^{u-1} N,$ 
moreover $h_k = \rho_k - \rho_{k-1} > h_k^* = \rho_k^* - \rho_{k-1}^*, k=1,\cdots,m.$

Hence the number $m$ of $\Delta^k,$ $k=0,1,\ldots,m$ is less than $N.$ 
We decompose each $\Delta^k$ into cubes or parallelepipeds $\Delta^k_{i_1,\ldots,i_l}$ in a 
way described above in the proof of the Theorem \ref{Theorem 3.1}(see the part
of constructing  a not necessarily continuous local spline). 
Clearly, the total number of $\Delta^k_{i_1,\ldots,i_l},$ 
$k=0,1,\ldots,m$ is equal to 
$
n \asymp n_0\asymp N^{v(l-1)},
$
where $n_0$ is the number of $\Delta^0_{i_1,\ldots,i_l}.$

Let $\Delta^k_{i_1,\ldots,i_l} = [b^k_{i_1}, b^k_{i_1+1}; \ldots; b^k_{i_l}, b^k_{i_l+1}],$ 
$k=0,1,\ldots,m.$
Introduce the functions
\begin{equation*}
\varphi^0_{i_1,\ldots,i_l}(t_1,\cdots,t_l) =\left \{
\begin{array}{cc}
A_0\frac{((t_1 - b^0_{i_1}) (b^0_{i_1+1} - t_1) \cdots (t_l - b^0_{i_l}) 
(b^0_{i_l+1} - t_l))^s}{h_0^{(2l-1)s}} N^{v\gamma} \ln^{u-1} N, \\ \,   t \in \Delta^0_{i_1,\ldots,i_l}, \ ,
0 ,\ ,  t \in \Omega \backslash \Delta^0_{i_1,\ldots,i_l};
\end{array}\right.
\end{equation*}
\begin{equation*}
\varphi^k_{i_1,\ldots,i_l}(t_1,\cdots,t_l) =\left \{
\begin{array}{ccc}
A_k\frac{((t_1 - b^k_{i_1}) (b^k_{i_1+1} - t_1) \cdots (t_l - b^k_{i_l})
 (b^k_{i_l+1} - t_l))^s}{h_k^{(2l-1)s} \rho_k^{\gamma}}, \\ \, t \in \Delta^k_{i_1,\ldots,i_l}, 
0 , \, t \in \Omega \backslash \Delta^k_{i_1,\ldots,i_l},
\end{array}\right.
\end{equation*}
$k=1,2,\ldots,m.$
Constants $A_k,$ $k=0,1,\ldots,m,$ are chosen such that
$
|D^s\varphi^0_{i_1,\cdots,i_l}|\le N^{v\gamma}\ln^{u-1}N,
$
$
|D^s\varphi^k_{i_1,\cdots,i_l}|\le 1/\rho_k^\gamma.
$
Obviously, such constants exist and do not depend on $N, u, \gamma.$ 

Let estimate the maximum values of $\varphi^k_{i_1,\ldots,i_l}(t),$ $k=0,1,\ldots,m.$
Clearly\\
$
\varphi^0_{i_1,\ldots,i_l}(t)  \geq c h_0^s  N^{v\gamma} \ln^{u-1} N = c N^{-v(s-\gamma)} \ln^{u-1}N = 
c N^{-s}\ln^{u-1}N,
$\\
$
\varphi^k_{i_1,\ldots,i_l}(t)  \geq c h_k^s/\rho_k^\gamma = c N^{-s}\ln^{u-1}N , h_k=\rho_k - \rho_{k-1},\ \ \ k=0,1,\ldots,m.
$

Let $\xi(t)$ be a linear combination
$
\xi(t) = \sum\limits_{k, i_1,\ldots,i_l} C^k_{i_1,\ldots,i_l} \varphi^k_{i_1,\ldots,i_l}(t),
$
where $|C^k_{i_1,\ldots,i_l}| \leq 1.$
Here the summation is taken over all domains $\Delta^k_{{i_1,\ldots,i_l}}$ of  $\Omega.$ 

Repeating the arguments presented in  \cite{[2],[3],[24]}, we have  
$
\delta_n(\bar Q^u_{r\gamma}(\Omega,1)) \geq c N^{-s}\ln^{u-1} N =
c n^{-(s-\gamma)/(l-1)}\ln^{u-1}n.
$
\end{proof}

\begin{rmk}
The estimate 
$
d_n(\bar Q^u_{r\gamma}(\Omega,1),C) \geq c n^{-(s-\gamma)/(l-1)}
\ln^{u-1} n
$
follows from the definition of $\xi(t)$ and Lemma \ref{Lemma 1.1}.
\end{rmk}

Let $v> l/(l-1).$
First, we construct a local spline not necessary continuous  which approximates
the functions of $\bar Q^u_{r \gamma}(\Omega,1)$ for $v>l/(l-1)$   
and has the error not exceeding  $c(\ln^u n)n^{-(s-\gamma)/(l-1)}.$
Afterwards we construct a continuous local spline having the same error of approximation. 

When  constructing  a local spline we employ the same process as used in the proof of Theorem \ref{Theorem 3.1}
(see the part of constructing  a not necessarily continuous local spline). 
We define  the domains  $\Delta^k$ and partition them into $\Delta^k_{i_1,\cdots,i_l,}$  $k=0,1,\ldots,N-2,$ 
in a similar way we did for  $v \le l/(l-1).$ 

Clearly, the number $n$ of $\Delta^k_{i_1,\cdots,i_l}$ is estimated by
\begin{equation}
n\asymp N^{v(l-1)}.
\label{(3.18)}
\end{equation}

The  polynomial
$P_{s,\ldots,s}(f;\Delta^{k}_{i_1,\cdots,i_l})$
interpolates $f$  in $\Delta^{k}_{i_1,\cdots,i_l},$ $k=0,1,\cdots,N-1.$
Hence  the  local spline $f_N$ is composed of the polynomials
$P_{s,\ldots,s}(f;\Delta^k_{i_1,\ldots,i_l}),$ $k=0,1,\ldots,N-1.$

It is easy to see that for $1 \leq k \leq N-1$ the following estimate holds
\begin{equation}
\|f - f_N\|_{C(\Delta^k_{i_1,\ldots,i_l})} \leq cN^{-s}(\ln N)^{u-1}.
\label{(3.19)}
\end{equation}

Indeed,
\begin{equation*}
\|f - f_N\|_{C(\Delta^k_{i_1,\ldots,i_l})} \leq
ch^s_k \frac{|\ln (\frac kN)^v|^{u-1}}{((k/N)^v)^\gamma} 
\leq ch^s_k \left(\frac{N}{k}\right)^{v\gamma}(\ln N)^{u-1} =
\frac{c}{N^s}(\ln N)^{u-1}.
\end{equation*} 

Let $k=0.$ Without loss of generality we demonstrate our computations in \\ 
$\Delta^0_{0,\ldots,0} = [-1,t_1; -1,t_1; \ldots; -1,t_1],$ where
$t_1 = -1 + \left(\frac{1}{N}\right)^v .$
Using Taylor's expansion (\ref{(3.6)}) we obtain
\begin{equation*}
\|f - f_N\|_{C(\Delta^0_{0,\ldots,0})} \leq
c\lambda^l_s E_{r-1,\ldots,r-1}(f,\Delta^0_{0,\ldots,0}) \leq
\end{equation*}
\begin{equation*}
\leq c\max\limits_{t \in \Delta^0_{0,\ldots,0}}
\left|\sum\limits_{|k|=r}\frac{1}{k!}\int\limits^1_0(1-\tau)^{r-1}(t-t^0)^k
(1+|\ln^u d(-1+\tau(t_k+1)), \Gamma)|)d\tau \right|\leq
\end{equation*}
\begin{equation*}
\leq c\max\limits_{t \in \Delta^0_{0,\ldots,0}}
\left|\sum\limits_{|k|=r}\frac{1}{k!}\int\limits^1_0(1-\tau)^{r-1}(t-t^0)^k
|\ln^u \tau(t_k+1)|)d\tau| \leq ch^r_0|\ln^u h_0\right| \leq
 c \frac{\ln^u N}{N^s},
\end{equation*}
where $t^0=(-1,\ldots,1).$

From the previous estimate and the equalities (\ref{(3.18)}) we have
$
\|f - f_N\|_{C(\Omega)} \leq c N^{-s}\ln^u N
\leq c n^{-(s-\gamma)/(l-1)}\ln^u n.
$

To construct the continuous local spline $f_N^*$ approximating 
$\bar Q^u_{r\gamma}(\Omega,1)$ for $v>l/(l-1)$ and having the error 
$c(\ln^u n) n^{-(s-\gamma)/(l-1)}$,
we employ all above constructions for the continuous local spline $f_N^*$ 
approximating $\bar Q^u_{r\gamma}(\Omega,1)$ when $v \le l/(l-1),$ cf. the proof of \ref{Theorem 3.1}.

Thus, for $0 \leq k \leq N-1,$ the following estimates hold
\begin{equation}
\|f - f_N^*\|_{C(\Delta^k_{i_1,\ldots,i_l})} \leq c N^{-s}(\ln N)^{u-1}, k=1,\cdots,N-1,
\label{(3.20)}
\end{equation}
\begin{equation}
\|f - f_N^*\|_{C(\Delta^0_{0,\ldots,0})} \leq c h^r_0|\ln^u h_0| \leq c N^{-s}\ln^u N.
\label{(3.21)}
\end{equation}

From the previous estimates and the equality (\ref{(3.18)}) we have
$
\|f - f_N^*\|_{C(\Omega)} \leq c N^{-s}\ln^u N
\leq c n^{-(s-\gamma)/(l-1)}\ln^u n.
$

Since the number of the nodes used in construction of the continuous local spline $f_N^*$ is
$s^l$ in each $\Delta^k_{i_1,\ldots,i_l},$ $k=0,1,\ldots,N-1,$
we state  the following

\begin{theorem} \rm
\label{Theorem 3.3} Let $\Omega = [-1,1]^l,$ $l \geq 2,$ $u = 1,2,\cdots,,$
 $v=s/(s-\gamma), v > l/(l-1).$ Then the  estimate 
$
 d_n(\bar Q^u_{r \gamma}(\Omega,1)) \le c n^{-(s-\gamma)/(l-1)}\ln^u n 
$
holds.
\end{theorem}

To estimate the Kolmogorov widths $ d_n(Q_{r\gamma}^u(\Omega, 1),C)$ for $u=1,2,\ldots,$ we use 
Definition \ref{Definition 1.7} for $Q^u_{r\gamma}(\Omega, 1)$ and note that $\gamma = s-r-1+\mu, \mu =1+\gamma - 
\lceil \gamma \rceil$.

\begin{theorem} \rm
\label{Theorem 3.4} Let $\Omega = [-1,1]^l,$ $l \geq 2,$
$ u=1,2,\cdots, $ $v=s/(s-\gamma).$
Then
\begin{equation}
 d_n( Q^u_{r \gamma}(\Omega,1),C) \le c
n^{-s/l}
\label{(3.22)}
\end{equation}
${\rm if}  \quad v<l/(l-1).$
\begin{equation}
 d_n(Q^u_{r\gamma}(\Omega,1),C)\le cn^{-s/l} (\ln n)^{us/(r+1-\mu)}
\label{(3.23)}
\end{equation}
${\rm for} \quad  lu/(r+1-\mu)\ge ul/s+1,$
\begin{equation}
 d_n(Q^u_{r\gamma}(\Omega,1),C)\le c n^{-s/l}(\ln n)^{(ul+s)/l}
\label{(3.24)}
\end{equation}
${\rm for} \quad lu/(r+1-\mu) < ul/s+1, \quad  {\rm if} \quad v=l/(l-1).$
\end{theorem}

\begin{proof} The proof of the theorem is similar to the proof of Theorem \ref{Theorem 3.1}.
First, we construct a local not necessarily continuous spline which approximates the functions of 
the class $Q^u_{r \gamma}(\Omega,1))$ and has the error given in the right-hand sides of 
(\ref{(3.22)}) -- (\ref{(3.24)}).  Afterwards we construct a continuous local spline having 
the same error of approximation. 

We decompose the domain $\Omega$ into subdomains
$\Delta^k_{i_1,\ldots,i_l;j_1,\cdots,j_l},$ $k=0,1,\ldots,N-1$ following the procedure 
which was described more than once in this section. (For instance, 
see the proof of Theorem \ref{Theorem 3.1}).In doing so, we divide each  
edge of $\Delta^k_{i_1,\ldots,i_l},$  into $M_k$ equal subintervals, 
$
M_k = 
\lceil (\ln N)^{u/(r+1-\mu)}\rceil,  k=0, \ \ 
M_k=\lceil (\ln (N/k))^{u/s}\rceil,  k=1,2,\ldots, N-1
$
and  pass the planes parallel to the coordinate planes throught the points of division.
To interpolate $f$ in each of the obtained cubes or parallelepipedes  
$\Delta^k_{i_1,\ldots,i_l;j_1,\ldots,j_l}$ we use  
the polynomial $P_{s,\ldots,s}(f,\Delta^k_{i_1,\ldots,i_l;j_1,\ldots,j_l})$
described in this section.  
Hence  the  local spline $f_N$ is composed of the polynomials
$P_{s,\ldots,s}(f;\Delta^k_{i_1,\ldots,i_l;j_1,\cdots,j_l}),$ $k=0,1,\ldots,N-1.$

\begin{rmk} We will use polynomials $P_{s,\ldots,s}(f,\Delta^k_{i_1,\ldots,i_l;j_1,\ldots,j_l})$ when $s \ge r+2,$
and polynomials  $P_{s+1,\ldots,s+1}(f,\Delta^k_{i_1,\ldots,i_l;j_1,\ldots,j_l})$ when $s=r+1.$ 
Without loss of generality we demonstrate our computations when $s \ge r+2.$
\end{rmk}

Estimating  an approximation $f_N$ to  
$f\in  Q^u_{r\gamma}(\Omega,1)$ we obtain for $1\leq k\leq N-1$
\begin{equation*}
\|f-P_{s,\ldots,s}
(f,\Delta^k_{i_1,\ldots,i_l;j_{1,\ldots,j_l}})\|_{C(\Delta^k_{i_1,\ldots,i_l;
j_1,\ldots,j_l})} \leq
\end{equation*}
\begin{equation}
\leq c\left(\left(\left(\frac{k+1}{N}\right)^v-\left(\frac{k}{N}\right)^v\right)
\frac{1}{(\ln \frac{N}{k})^{u/s}}\right)^s\frac{(1+|\ln (\frac{k}{N})^v|)^{u}}
{\left(\frac{k}{N}\right)^{v\gamma}}\leq c\frac{1}{N^{s}}.
\label{(3.25)}
\end{equation}

If $k=0$, then
$
\|f-P_{s,\ldots,s}
(f,\Delta^0_{i_1,\ldots,i_l;j_1,\ldots,j_l})\|_{C(\Delta^0_{i_1,\ldots,i_l;
j_1,\ldots,j_l})} \leq \\
\leq cE_{r+1,\ldots,r+1}(f,\Delta^0_{i_1,\ldots,i_l;j_1,\ldots,j_l})\lambda_s^l.
$

Using  Taylor's expansion with the remainder in integral form
we have
\begin{equation*}
E_{r+1,\ldots,r+1}(f,\Delta^0_{i_1,\ldots,i_l;j_1,\ldots,j_l})
\leq ch_{00}^{r+1-\mu}\int\limits_0^1(1-\tau)^{r-1}(1+|\ln^u(\tau h_{00})|)d\tau\leq
\end{equation*}
\begin{equation*}
\leq ch_{00}^{r+1-\mu}\ln^uh_{00}\leq c N^{-v(r+1-\mu)}=c N^{-s},
\end{equation*}
where $h_{00}=h_0/M_0,\ \ h_0=(1/N)^v.$

Hence,
\begin{equation}
\|f-P_{s,\ldots,s}(f,\Delta^0_{i_1,\ldots,i_l;
j_1,\ldots,j_l})\|_{C(\Delta^0_{i_1,\ldots,i_l;j_1,\ldots,j_l})}
\leq c N^{-s}. 
\label{(3.26)}
\end{equation}

Combining (\ref{(3.25)}) -- (\ref{(3.26)}) gives
\begin{equation}
\|f-f_N\|\leq c N^{-s}.
\label{(3.27)}
\end{equation}

Now we estimate the number of nodes used in constructing  $f_N.$
As in the proof of Theorem \ref{Theorem 3.1}, we study two cases i) $v<l/(l-1)$ and ii) $v=l/(l-1)$.

{ i).} Let $v<l/(l-1)$. The estimate follows immediately from the chain of inequalities
\begin{equation*}
n\leq m \sum\limits^{N-1}_{k=1} \left(\frac{2-2(\frac{k}{N})^v}{(\frac{k+1}{N})^v - (\frac{k}{N})^v}\right)^{l-1} M_k^l+
2mN^{v{(l-1)}} \lceil\ln N\rceil^{lu/(r+1-\mu)} \leq
\end{equation*}
\begin{equation*}
\leq c N^{v(l-1)} (\ln N)^{lu/(r+1-\mu)} + c\sum \limits_{k=1}^{N-1}
 \left( \frac{2N^v-2k^v}{v(k+\theta)^{v-1}} \right)^{l-1}\left(1+\left(  \ln \frac{N}{k}\right)^{\frac{u}{s}} \right)^l\le
\end{equation*}
\begin{equation*}
\leq c N^{v(l-1)} (\ln N)^{lu/(r+1-\mu)}+c\sum \limits_{k=1}^{N-1}
\frac{N^{v(l-1)}}{k^{(v-1)(l-1)}}\left(\left(( \ln \frac{N}{k} \right)^{\frac{ul}{s}}+1\right)\le cN^l,
\end{equation*}
where $m$ is the number of faces of  $\Delta^k_{i_1,\cdots,i_l}.$

One can obtain  that $n \geq cN^l$ is true in a similar way. Therefore,
\begin{equation}
n = cN^l.
\label{(3.28)}
\end{equation}

The inequalities (\ref{(3.27)}) -- (\ref{(3.28)}) yield
$
\|f - f_N\| \leq c n^{-s/l},
$
where $n $ is the number of nodes of the local spline.

{ ii).} Let $v=l/(l-1).$
As we have already derived for $v<l/(l-1),$
the upper bound follows immediately from the chain of inequalities
\begin{equation*}
n\leq m\sum\limits^{N-1}_{k=1}\left( \frac{2-2(\frac{k}{N})^v}{(\frac{k+1}{N})^v-
(\frac{k}{N})^v} \right)^{l-1}M_k^l+2mN^{v(l-1)}\lceil\ln N\rceil^{lu/(r+1-\mu)}\leq
\end{equation*}
\begin{equation*}
\leq cN^l(\ln N)^{lu/(r+1-\mu)}+cN^l(\ln N)^{(ul/s)+1}.
\end{equation*}

It remains to express $N$ in terms of $n.$ It is necessary to study two cases:

{ i).}the estimate $N\le n^{1/l}/(ln n)^{u/(r+1-\mu)}$ holds, if $lu/(r+1-\mu) \ge ul/s+1$;

{ ii).}$N\le n^{1/l}/(ln n)^{(ul+s)/(sl)}$ holds, if $lu/(r+1-\mu) < ul/s+1$.

Thus, for  $lu/(r+1-\mu)\ge ul/s+1$
\begin{equation}
\|f-f_N\|\le cn^{-s/l}(\ln n)^{us/(r+1-\mu)};
\label{(3.29)}
\end{equation}
${\rm for} \quad lu/(r+1-\mu) < ul/s+1$
\begin{equation}
\|f-f_N\|\le cn^{-s/l}(\ln n)^{(ul+s)/l}.
\label{(3.30)}
\end{equation}

To obtain the upper estimate of the Kolmogorov widths of the functional class
$Q^u_{r \gamma}(\Omega,1)),$ we construct a continuous local spline 
which has the error  given in the right-hand sides of (\ref{(3.29)}), (\ref{(3.30)}).  
For this purpose, we repeart the construction of the continuous local
spline provided in the proof of Theorem \ref{Theorem 3.1}. 
One can show that the continuous local spline that approximates  
the functions of $Q^u_{r \gamma}(\Omega,1)$ has the error given in the 
right-hand sides of  (\ref{(3.29)}), (\ref{(3.30)}). 
\end{proof}

\begin{theorem} \rm
\label{Theorem 3.5} Let $\Omega = [-1,1]^l,$ $l \geq 2,$
$u=1,2,\cdots,$ $v=s/(s-\gamma),$ $v<l/(l-1).$ 
The estimate 
$
\delta_n( Q^u_{r \gamma}(\Omega,1)) \geq c n^{-s/l}
$
holds.
\end{theorem}

\begin{proof} 
It is easy to see that 
$Q_{r \gamma}(\Omega,1) \subset Q^u_{r \gamma}(\Omega,1)$. 
From Theorem \ref{Theorem 1.3}
the estimate $\delta_n( Q_{r \gamma}(\Omega,1)) \geq c n^{-s/l}$
follows. Therefore  $\delta_n( Q^u_{r \gamma}(\Omega,1)) \geq c n^{-s/l}.$ 
\end{proof}

\begin{theorem} \rm
\label{Theorem 3.6} 
 Let $\Omega = [-1,1]^l,$ $l \geq 2,$ $u=1,2,\cdots,$ $v=s/(s-\gamma),$ $v= l/(l-1).$ Then
$
\delta_n(Q^u_{r\gamma}(\Omega, 1)) \geq c n^{-s/l}(\ln n)^{u+s/l}.
$
\end{theorem}

\begin{proof} We decompose the domain $\Omega$ into subdomains
$\Delta^k_{i_1,\ldots,i_l; j_1,\ldots,j_l},$ $k=0,1,\ldots,N-1$ following the procedure
which was described in the proof of the
Theorem \ref {Theorem 3.1} (see the part of Babenko widths estimates for $v=l/(l-1)$).

Now, we introduce
$
M_k = 
\lceil(\ln N)^{u/s}\rceil, \, k=0;\ \  M_k=
\lceil(\ln (N/k))^{u/s}\rceil, \, k=1,2,\ldots,N-1. 
$

Let estimate the number   $\Delta^k_{i_1,\ldots,i_l; j_1,\ldots,j_l},$  $k=0,1,\cdots,N-1.$ Clearly
$$
n \asymp m\sum\limits^{N-1}_{k=1}
\left(\frac{2-2\left(\frac{k}{N}\right)^v}{\left(\frac{k+1}{N}\right)^v - \left(\frac{k}{N}\right)^v}\right)^{l-1} 
M^l_k + 2m N^{v(l-1)}(\ln N)^{lu/s} \asymp
$$
$$
\asymp N^l(\ln N)^{lu/s} + N^{l-1}\int\limits^N_1
\frac{N}{x}\left(\ln \frac{N}{x}\right)^{ul/s}dx \asymp N^l (\ln N)^{1+ul/s}.
$$

Let $\Delta^k_{i_1,\ldots,i_l; j_1,\ldots,j_l} = [b_{i_1,j_1}, b_{i_1, j_1+1}; \cdots; b_{i_l,j_l}, b_{i_l, j_l+1}].$
Introduce the functions
\begin{equation*}
\varphi^k _{i_1,\ldots,i_l; j_1,\ldots,j_l}(t) = 
\end{equation*}
\begin{equation*}
=\begin{cases}
A_k\frac
{((t_1 - b_{i_1, j_1})(b_{i_1, j_1+1} - t_1) \cdots (t_l - b_{i_l, j_1})(b_{i_l, j_1+1} - t_l) )^s}{(h_k/M_k)^{(2l-1)s}((k+1)/N)^{v\gamma}}
\left(1+ \left|\ln^u \left(\frac{k+1}{N}\right)^v\right|\right),\\  t\in \Delta^k_{i_1,\cdots,i_l;j_1,\cdots,j_l},
0,  t\in \Omega\setminus \Delta^k_{i_1,\cdots,i_l;j_1,\cdots,j_l},
\end{cases}
\end{equation*}
$h_k=((k+1)/N)^v-(k/N)^v, k=0,1,\ldots,N-1.$ Constants $A_k,$ $k=0,1,\ldots,N-1,$ are chosen such that
$$
\left| D^s\varphi^k_{i_1,\cdots,i_l;j_1,\cdots,j_l}\right| \leq 
\frac{1}{((k+1)/N)^{v\gamma}}
\left(1+\left|\ln^u\left(\frac{k+1}{N}\right)^v\right|\right).
$$
Obviously, such constants exist and do not depent on $N, u, \gamma.$ 

Let estimate the maximum values of  $\varphi^k _{i_1,\ldots,i_l; j_1,\ldots,j_l}(t).$ 
Clearly,
$$
\varphi^k _{i_1,\ldots,i_l; j_1,\ldots,j_l}(t)  \geq A_k\left(\frac{h_k}{M_k}\right)^s \left(\frac{N}{k+1}\right)^{v\gamma}
\left(1+|\ln^u \left(\frac{k+1}{N}\right)^v|\right) =
$$
$$
= A_k \left(\left(\frac{k+1}{N}\right)^v - \left(\frac{k}{N}\right)^v\right)^s
\frac{1}{\left(\left(\ln \frac{N}{k}\right)^{u/s} +1\right)^s} \left(\frac {N}{k+1}\right)^{v\gamma}
\left(1+\left|\ln^u \left(\frac{k+1}{N}\right)^v\right|\right) \ge \frac{c}{N^s}
$$
for $k=1,2,\ldots,N-1;$
$$
\varphi^0_{i_1,\ldots,i_l; j_1,\ldots,j_l}(t)  \geq A_0
\left(\frac{h_0}{M_0}\right)^s \left(\frac{N}{1}\right)^{v\gamma}\left(1+\left|\ln^u \left(\frac{1}{N}\right)^v\right|\right) \geq c
\frac{1}{N^s}.
$$

Let $\xi(t)$ be a linear combination 
$
\xi(t) = \sum\limits_{k, i_1,\ldots,i_l; j_1,\ldots,j_l} C^k_{i_1,\ldots,i_l; j_1,\ldots,j_l}
\varphi^k_{i_1,\ldots,i_l; j_1,\ldots,j_l}(t),
$\\
where $|C^k_{i_1,\ldots,i_l; j_1,\ldots,j_l}| \leq 1.$ Here the summation is taken over all domains 
$\Delta^k_{i_1,\ldots,i_l; j_1,\ldots,j_l}$ of $\Omega.$ 

Repeating the arguments presented in \cite{[2],[3],[24]} we have 
$
\delta_n(Q^u_{r \gamma},(\Omega,1) ) \geq c n^{-s/l}(\ln n)^{u +s/l}.
$
\end{proof}

\begin{rmk}
The estimate 
$
d_n( Q^u_{r\gamma}(\Omega,1),C) \geq c n^{-s/l}(\ln n)^{u +s/l} 
$
follows from the definition of $\xi(t)$ and Lemma \ref{Lemma 1.1}.
\end{rmk}

Combining the statements of  Theorems \ref {Theorem 3.4} -- \ref {Theorem 3.6} we have the following  

\begin{theorem} \rm
\label{Theorem 3.7} Let 
$\Omega = [-1,1]^l,$ $l \geq 2,$ $u=1,2,\cdots,$ $v=s/(s-\gamma).$ Then\\
$
\delta_n(Q^u_{r \gamma} (\Omega,1)) \asymp 
d_n(Q^u_{r \gamma} (\Omega,1),C) \asymp 
n^{-s/l}
$
if $v < l/(l-1);$\\
$
cn^{-s/l}(\ln n)^{(ul +s)/l}\le\delta_n(Q^u_{r \gamma} (\Omega,1)) \le 
2d_n(Q^u_{r \gamma} (\Omega,1), C) \le c n^{-s/l}(\ln n)^{u s/(r+1-\mu)} 
$\\
for $lu/(r+1-\mu) \geq ul/s + 1,$\\
$
\delta_n(Q^u_{r \gamma} (\Omega,1)) \asymp  
d_n(Q^u_{r \gamma} (\Omega,1), C) \asymp  n^{-s/l}(\ln n)^{(ul +s)/l}
$
for $lu/(r+1-\mu) < ul/s + 1,$\\ if $v = l/(l-1).$
\end{theorem}

\begin{theorem} \rm
\label{Theorem 3.8} 
Let  $\Omega = [-1, 1]^l,$ $l \geq 2,$ $v = s/(s-\gamma),$ $v > l/(l-1).$ Then
$
\delta_n(Q^u_{r\gamma}(\Omega, 1)) \geq c n^{-(s-\gamma)/(l-1)} \ln^{u} n.
$
\end{theorem}

\begin{proof}
The proof of the theorem is similar to the proof of Theorem \ref{Theorem 3.2}.
The difference is that we define  the function $\varphi^0_{i_1,\ldots,i_l}(t)$  by
\begin{equation*}
\varphi^0_{i_1,\ldots,i_l}(t)=
\begin{cases}
A_0\frac{((t_1 - b^0_{i_1}) (b^0_{i_1+1} - t_1) \cdots (t_l - b^0_{i_l}) (b^0_{i_l+1} - t_l))^s}{h_0^{(2l-1)s}} N^{v\gamma}
\ln^{u} N ,\, \ t \in \Delta^0_{i_1,\ldots,i_l}, \\
0 ,\,  \, t \in \Omega \backslash \Delta^0_{i_1,\ldots,i_l}.
\end{cases}
\end{equation*}
\end{proof}

\begin{theorem} \rm
\label{Theorem 3.9} 
Let $\Omega = [-1, 1]^l,$ $l \geq 2,$ $v = s/(s-\gamma),$ $v > l/(l-1).$ Then
$
 d_n(Q^u_{r\gamma}(\Omega,1),C) \le c  n^{-(s-\gamma)/(l-1)} \ln^{u} n.
$
\end{theorem}
\begin{proof}
We decompose the domain $\Omega$ in to subdomains $\Delta^k_{i_1,\dots,i_l}, k=0,1,\dots,N-1,$ following the procedure which was described in 
the proof of Theorem \ref{Theorem 3.3}.

Clearly, the number of $\Delta^k_{i_1,\cdots,i_l}$ is estimated by
\begin{equation}
n\asymp N^{v(l-1)}.
\label{(3.31)}
\end{equation}

The  polynomial
$P_{s,\ldots,s}(f;\Delta^{k}_{i_1,\cdots,i_l})$
interpolates $f$  in $\Delta^{k}_{i_1,\cdots,i_l},$ $k=0,1,\cdots,N-1.$
Hence  the  local spline $f_N$ is composed of the polynomials
$P_{s,\ldots,s}(f;\Delta^k_{i_1,\ldots,i_l}),$ $k=0,1,\ldots,N-1.$

It is easy to see that for $1 \leq k \leq N-1$ the following estimate holds
\begin{equation}
\|f - f_N\|_{C(\Delta^k_{i_1,\ldots,i_l})} \leq cN^{-s}(\ln N)^{u}.
\label{(3.32)}
\end{equation}

Indeed
$\|f - f_N\|_{C(\Delta^k_{i_1,\ldots,i_l})} \leq
ch^s_k \frac{|\ln (\frac kN)^v|^{u}}{((k/N)^v)^\gamma} 
\leq ch^s_k \left(\frac{N}{k}\right)^{v\gamma}(\ln N)^{u} =
\frac{c}{N^s}(\ln N)^{u}.
$

Let $k=0.$ Without loss of generality we demonstrate our computations in \\ 
$\Delta^0_{0,\ldots,0} = [-1,t_1; -1,t_1; \ldots; -1,t_1],$ where
$t_1 = -1 + \left(\frac{1}{N}\right)^v .$
Using Taylor's expansion (\ref{(3.6)}) we obtain
\begin{equation*}
\|f - f_N\|_{C(\Delta^0_{0,\ldots,0})} \leq
c\lambda^l_s E_{r,\ldots,r}(f,\Delta^0_{0,\ldots,0}) \leq
\end{equation*}
\begin{equation*}
\leq c\max\limits_{t \in \Delta^0_{0,\ldots,0}}
\left|\sum\limits_{|k|=r+1}\frac{1}{k!}\int\limits^1_0(1-\tau)^{r}(t_k+1)^k
\frac{(1+|\ln^u d(-1+\tau(t_k+1)), \Gamma)|)}{(d(-1+\tau(t_k+1)), \Gamma))^{1-\zeta}}
d\tau \right|\leq c \frac{\ln^u N}{N^s}.
\end{equation*}

From the previous estimate and the equality (\ref{(3.31)}) one obtains
$
\|f - f_N\|_{C(\Omega)} \leq c N^{-s}\ln^u N
\leq c n^{-(s-\gamma)/(l-1)}\ln^u n.
$

To construct the continuous local spline $f_N^*$ approximating 
$ Q^u_{r\gamma}(\Omega,1)$ for $v>l/(l-1)$ with the error 
$c(\ln^u n) n^{-(s-\gamma)/(l-1)}$,
we employ all above constructions for the continuous local spline $f_N^*$ 
approximating $\bar Q^u_{r\gamma}(\Omega,1)$ when $v \ge l/(l-1),$ cf. the proof of \ref{Theorem 3.3}.

Thus, for $0 \leq k \leq N-1,$ the following estimates hold
\begin{equation}
\|f - f_N^*\|_{C(\Delta^k_{i_1,\ldots,i_l})} \leq c N^{-s}(\ln N)^{u}, k=1,\cdots,N-1,
\label{(3.33)}
\end{equation}
\begin{equation}
\|f - f_N^*\|_{C(\Delta^0_{0,\ldots,0})} \leq c h^{r+1-\mu}_0|\ln^u h_0| \leq c N^{-s}\ln^u N.
\label{(3.34)}
\end{equation}

From the previous estimates and the equality (\ref{(3.31)}) we have$
\|f - f_N^*\|_{C(\Omega)} \leq c N^{-s}\ln^u N
\leq c n^{-(s-\gamma)/(l-1)}\ln^u n.
$

Since the number of the nodes used to construct the continuous local spline $f_N^*$ is
$s^l$ in each $\Delta^k_{i_1,\ldots,i_l},$ $k=0,1,\ldots,N-1,$
we state  the following estimate
$
 d_n(Q^u_{r\gamma}(\Omega,1),C) \le c  n^{-(s-\gamma)/(l-1)} \ln^{u} n.
$
\end{proof}

Combining Theorem \ref{Theorem 3.8} and Theorem \ref{Theorem 3.9} leads us to the following 
\begin{theorem} \rm
\label{Theorem 3.10} 
Let $\Omega = [-1, 1]^l,$ $l \geq 2,$ $u=1,2,\cdots,$ $v = s/(s-\gamma),$ $v > l/(l-1).$ Then
$
\delta_n(Q^u_{r\gamma}(\Omega, 1)) \asymp d_n(Q^u_{r\gamma}(\Omega,1),C)
\asymp  n^{-(s-\gamma)/(l-1)} \ln^{u} n.
$
\end{theorem}

\end{document}